[1994/12/01]
\documentclass[11pt]{amsart}
\usepackage{microtype, mathtools, mathrsfs, enumerate, dsfont, esvect}
\usepackage{verbatim}
\usepackage[linktocpage]{hyperref}
\usepackage[arrow,matrix]{xy}
\usepackage{amssymb}
\usepackage{url}
\chardef\bslash=`\\ 

\newtheorem{theorem}{Theorem}[section]
\newtheorem*{ta}{Theorem A}
\newtheorem*{tb}{Theorem B}
\newtheorem*{tc}{Theorem C}
\newtheorem*{td}{Theorem D}
\newtheorem{proposition}[theorem]{Proposition}

\newtheorem{lemma}[theorem]{Lemma}
\newtheorem{definition-lemma}[theorem]{Definition-Lemma}

\newtheorem{corollary}[theorem]{Corollary}

\theoremstyle{definition}

\newtheorem{definition}[theorem]{Definition}

\theoremstyle{plain}
\numberwithin{equation}{theorem}

\theoremstyle{remark}

\newtheorem{remark}[theorem]{Remark}
\newtheorem{example}[theorem]{Example}

\def\Frac{\operatorname{Frac}}

\DeclareMathOperator{\Der}{Der}

\newif\ifhascomments \hascommentstrue
\ifhascomments
  \newcommand{\dragos}[1]{{\color{red}[[\ensuremath{\bigstar\bigstar\bigstar} #1]]}}
  \newcommand{\matt}[1]{{\color{red}[[\ensuremath{\spadesuit\spadesuit\spadesuit} #1]]}}
\else
  \newcommand{\dragos}[1]{}
  \newcommand{\matt}[1]{}
\fi

\begin{document}
\title[Some topological criteria for the Poisson Dixmier-Moeglin equivalence]{Poisson Dixmier-Moeglin equivalence from a topological point of view}
\author[J.~Luo]{Juan Luo}
\address{Mathematics and Science College\\
Shanghai Normal University\\
Shanghai 200234\\
China
}
\email{luojuan@shnu.edu.cn} 

\author[X.~Wang]{Xingting Wang}
\address{Department of Mathematics\\
Howard University\\
2400 Sixth St. NW\\
Washington DC, 20059\\
USA}
\email{xingting.wang@howard.edu}

\author[Q.-S. Wu]{Quanshui Wu}
\address{School of Mathematical Sciences\\
Fudan University\\
Shanghai 200433\\
China}
\email{qswu@fudan.edu.cn}

\begin{abstract}
A complex affine Poisson algebra $A$ is said to satisfy the Poisson Dixmier-Moeglin equivalence if the Poisson cores of maximal ideals of $A$ are precisely those Poisson prime ideals that are locally closed in the Poisson prime spectrum ${\rm P. spec}\, A$ and if, moreover, these Poisson prime ideals are precisely those whose extended Poisson centers are exactly the complex numbers.  

In this paper, we provide some topological criteria for the Poisson Dixmier-Moeglin equivalence for $A$ in terms of the poset $({\rm P. spec}\, A,\subseteq)$ and the symplectic leaf or core stratification on its maximal spectrum.  In particular, we prove that the Zariski topology of the Poisson prime spectrum and of each symplectic leaf or core can detect the Poisson Dixmier-Moeglin equivalence for any complex affine Poisson algebra. Moreover, we generalize the weaker version of the Poisson Dixmier-Moeglin equivalence for a complex affine Poisson algebra proved in [J. Bell, S. Launois, O.L. S\'anchez, and B. Moosa, Poisson algebras via model theory and differential  algebraic geometry, J.~Eur.~Math.~Soc. (JEMS), 19(2017), no. 7, 2019--2049] to the general context of a commutative differential algebra.
\end{abstract}

\subjclass[2010]{
16D60, 
17B63,
13N15.
}

\keywords{Dixmier-Moeglin equivalence, commutative differential algebras, Poisson algebras, primitive ideals, prime spectrum}

\maketitle
\tableofcontents

\section{Introduction}
Let $k$ be a base field of characteristic zero. For a (left) noetherian $k$-algebra $S$, one of the most beautiful results in the direction of understanding the representation theory of $S$ is through the work of Dixmier \cite{Dix77} and Moeglin \cite{Moe80} when $S=U(\mathfrak g)$, the enveloping algebra of a finite dimensional complex Lie algebra $\mathfrak g$, by characterizing primitive ideals of $S$. More precisely, Dixmier and Moeglin showed the following types of prime ideals in $S$ coincide:
\begin{itemize}
\item the set of locally closed prime ideals (those constitute locally closed points in the prime spectrum of $S$);
\item the set of primitive ideals (those being annihilators of simple (left) modules over $S$); and
\item the set of rational prime ideals (those prime ideals $P$ in $S$ such that the center of the Goldie quotient ring of $S/P$ is algebraic over the base field $k$).
\end{itemize}
In their honor, nowadays we refer to $S$ as satisfying the {\it Dixmier-Moeglin equivalence} if the equivalence of three above properties holds for primes in the spectrum of $S$. Ever since, the Dixmier-Moeglin equivalence has been recognized as a very general phenomenon that happens for a wide range of algebras, see for example \cite{Von, Z, GoLet, BL14, BSM18, BRS10, GZ}. Very recently, it is proved in \cite{BWY19} that the Dixmier-Moeglin equivalence can be totally captured by the Zariski topology of the prime spectrum of $S$ under the cardinality assumption $\dim_k S<|k|$.

The notion of Poisson bracket, first introduced by Sim\'eon Denis Poisson, arises naturally in Hamiltonian mechanics and differential geometry. Poisson algebras have become deeply entangled with noncommutative geometry, integrable systems, topological field theory and representation theory of noncommutative algebras. They are essential in the study of the noncommutative discriminant \cite{BY18, NTY17} and representation theory of noncommutative algebras \cite{WWY17, WWY18}. In addition, there has been renewed interest in enveloping algebras of Poisson algebras \cite{LWZ15,LWZ17}.

The aim of this paper is to study an analogue of the Dixmier-Moeglin equivalence for complex affine Poisson algebras from a topological point of view. An {\it complex affine Poisson algebra} is a finitely generated commutative $\mathbb C$-algebra $A$ together with a bilinear map $\{-,-\}: A\times A \to A$ that is both a Lie bracket and a bi-derivation. The {\it Poisson center} of $A$ is the subalgebra of $A$ defined by $Z_P(A):=\{a\in A\,|\, \{a,b\}=0, \forall\, b\in A\}$. An ideal $I$ of $A$ is called a {\it Poisson ideal} if $\{I,A\}\subseteq I$. For each ideal $J$ of $A$, there is a largest Poisson ideal contained in $J$, which is called the {\it Poisson core} of $J$. The {\it Poisson primitive ideals} of $A$ are the Poisson cores of the maximal ideals in $A$. A {\it Poisson prime ideal} of $A$ is a prime ideal of $A$ that is also a prime ideal. The {\it Poisson prime spectrum} of $A$ is denoted by ${\rm P. spec}\, A$ consisting of all Poisson prime ideals in $A$, which is induced with the subspace topology from the prime spectrum ${\rm spec}\, A$. We say that $A$ satisfies the {\it Poisson Dixmier-Moeglin equivalence} provided the following types of Poisson prime ideals in ${\rm P. spec}\, A$ coincide:
\begin{itemize}
\item the set of locally closed Poisson prime ideals;
\item the set of Poisson primitive ideals; and 
\item the set of Poisson rational ideals (those Poisson prime ideals $I$ in $A$ such that $Z_P(\Frac(A/I))=\mathbb C$).
\end{itemize}

Some examples in which this equivalence holds have been given by \cite[Theorem 2.4, Proposition 2.13]{Oh99}. Later, this equivalence has been established for Poisson algebras with suitable rational torus actions by Goodearl \cite{Go06} in the general context of a commutative differential $k$-algebra and for any affine complex Poisson algebra with only finitely many Poisson primitive ideals by Brown and Gordon \cite[Lemma 3.4]{BrGor03}. On the other side, the tool of model theory and differential algebraic geometry has been recently employed to study the Poisson Dixmier-Moeglin equivalence for affine complex Poisson algebras; see references \cite{BLLM17,LL19}.

In this paper, our first main result is to assert that, for any complex affine Poisson algebra $A$, there is a purely topological characterization of the Poisson Dixmier-Moeglin equivalence. In particular, we show that the poset of Poisson prime ideals $({\rm P. spec}\, A,\,\subseteq)$ can detect whether or not the Poisson Dixmier-Moeglin equivalence holds.

\begin{ta}[Theorem \ref{thm:Poissonmain1}]
Let $A$ be a complex affine Poisson algebra. Then $A$ satisfies the Poisson Dixmier-Moeglin equivalence if and only if any Poisson prime ideal that has a set of minimal Poisson prime ideals over it with cardinality less than $|\mathbb C|$ has only finitely many minimal Poisson prime ideals.
\end{ta}

This result says that for complex affine Poisson algebras the underlying Zariski topology on the Poisson prime spectrum ``sees" the Poisson Dixmier-Moeglin equivalence. In Definition \ref{lem:equisep},  the notion of {\it $\kappa$-separability} is introduced to accommodate this new topological characterization mentioned above for a fixed cardinality $\kappa$ in an arbitrary Zariski space. In particular, it is proved that a point in a Zariski space is locally closed if and only if it is $\aleph_0$-separable (Lemma \ref{lem:Zlocallyclosed}). This concept is motivated by the work in \cite{BWY19} on Dixmier-Moeglin equivalence for left noetherian algebras over large base fields and we expect it will play a role in further study of the (Poisson) Dixmier-Moeglin equivalence.

One consequence of Theorem A is that the Poisson Dixmier-Moeglin equivalence can be detected locally in the Poisson prime spectrum with the Zariski topology since the notion of $\kappa$-separability is a local property according to Lemma \ref{lem:localpropertysep}.

\begin{tb}[Theorem \ref{thm:Poissonmain2}]
Let $A$ be a complex affine Poisson algebra. Suppose
$${\rm P. spec}\, A~=~\bigcup\, X_i$$ is a union of locally closed subsets with each $X_i$, when endowed with the subspace topology, homeomorphic to ${\rm P. spec}\, A_i$ for some complex affine Poisson algebra $A_i$. Then $A$ satisfies the Poisson Dixmier-Moeglin equivalence if and only if each $A_i$ satisfies the Poisson Dixmier-Moeglin equivalence.
\end{tb}

The relevance of this theorem is seen in the fact that the semiclassical limits of many quantum algebras have Poisson prime spectra of this form \cite{Go06}, although generally in these cases work of Goodearl \cite{Go06}, using additional information about the stratification coming from the rational action of an algebraic torus on the Poisson algebra, allows one to deduce that the Poisson Dixmier-Moeglin equivalence holds. Our Theorem B again strengthens the fact that for complex affine Poisson algebras an abstract stratification with parts homeomorphic to complex affine schemes of finite type immediately yields the Poisson Dixmier-Moeglin equivalence without having any underlying action of an algebraic group.

It is important to notice that the tool of stratification is a key ingredient in the important work of Goodearl and Letzter \cite{GoLet} in obtaining the Dixmier-Moeglin equivalence for many classes of quantum algebras and the corresponding Poisson Dixmier-Moeglin equivalence for their semiclassical limits \cite{Go06}. When we turn to the maximal (or prime) spectrum of a complex affine Poisson algebra $A$, there are several stratifications. One algebraic stratification is to use the Poisson core, where two maximal (or prime) ideals in $A$ are set to be in the same {\it symplectic core} if they share the same Poisson core. Our next result provides another topological characterization of the Poisson Dixmier-Moeglin equivalence for complex affine Poisson algebras in terms of their symplectic core stratification.

\begin{tc}[Theorem \ref{thm:algebraicDME}]
Let $A$ be a complex affine Poisson algebra. Then $A$ satisfies the Poisson Dixmier-Moeglin equivalence if and only if all symplectic cores in its maximal (or prime) spectrum are locally closed.
\end{tc}

Moreover in the case of complex affine Poisson algebras, we can consider the symplectic foliation on the underlying complex Poisson varieties in the sense of \cite{Wein83} for smooth Poisson manifolds and generally follow the work of \cite{BrGor03} for the singular case.  In \cite{BrGor03}, the Poisson bracket of a complex affine Poisson algebra $A$ is said to be {\it algebraic} if all the symplectic leaves are locally closed. Now suppose the Poisson bracket of $A$ is algebraic. It is further shown that the symplectic leaf stratification coincides with the symplectic core stratification \cite[Proposition 3.6(2)]{BrGor03}. Therefore, we know $A$ satisfies the Poisson Dixmier-Moeglin equivalence (Corollary \ref{cor:algebraicDME}). 

In general, it is difficult to show the Poisson bracket of $A$ is algebraic if there are infinitely many symplectic leaves. But when there are finitely many symplectic leaves, one result of  \cite[Proposition 3.7(1)]{BrGor03} ensures that the Poisson bracket of $A$ is algebraic. Our last result provides a $G$-equivariant version of the aforementioned result when we have a rational action of an algebraic group $G$ on $A$ with its induced action on the symplectic leaves and cores in the corresponding stratifications.

\begin{td}[Theorem \ref{thm:GalgebraDME}]
Let $G$ be an algebraic group acting rationally by Poisson automorphism on a complex affine Poisson algebra $A$. If the set of $G$-orbits of symplectic leaves (or symplectic cores) is finite, then $A$ satisfies the Poisson Dixmier-Moeglin equivalence.
\end{td}

The outline of the paper is as follows. In Section \ref{S:sep}, we discuss the topological criterion used in Theorem A and study some of its basic properties in an arbitrary Zariski space. In Section \ref{S:D}, we review some facts of a commutative differential $k$-algebra and state the corresponding Dixmier-Moeglin equivalence. Later in Section \ref{S:M}, we develop our main theorems in the general context of a commutative differential $k$-algebra over a large base field $k$, where we prove Theorem A and Theorem B  within that context. In Section \ref{S:S}, we extend the symplectic core stratification to a commutative differential $k$-algebra and study its relations with the Dixmier-Moeglin equivalence together with a possible rational action of an algebraic group. Finally in Section \ref{S:P}, we apply our previous results to complex affine Poisson algebras and prove Theorem C and Theorem D there.
\centerline{}
\centerline{{\it Throughout the paper, $k$ will denote a field of characteristic zero.}}

\section{$\kappa$-separability for Zariski spaces and posets}\label{S:sep}
A {\it Zariski space} is a topological space $X$ that every open subset is quasicompact and every nonempty closed irreducible subset $Y$ is the closure of a unique point $p$, where $p$ is called the {\it generic point} of $Y$. 
Every every open subset of $X$ is quasicompact is equivalent to that $X$ is a noetherian space, that is, $X$ satisfies the ascending chain condition for open subsets.  For instance, the prime spectrum ${\rm spec}\, R$ of any noetherian ring with identity is a Zariski space.

This following equivalent descriptions of locally closed subsets in a topological space are standard.

\begin{definition-lemma}\label{dl:clopen}
Let $X$ be a topological space. For any subset $S\subseteq X$, the following are equivalent:
\begin{enumerate}
\item $S$ is the intersection of an open and a closed subsets of $X$.
\item Every point $s\in S$ has an open neighborhood $Y$ in $X$ such that $S\cap Y$ is closed in $Y$.
\item $S$ is open in its closure $\overline{S}$ in $X$.
\end{enumerate}
If the above conditions hold, we say that the subset $S$ is {\it locally closed} in $X$.
\end{definition-lemma}

Throughout the paper, we will employ the above equivalent definitions of locally closed subset implicitly.

\begin{definition}
Let $X$ be a Zariski space, and $\kappa$ be a cardinal number. A point $p\in X$ is said to be {\it $\kappa$-separable} if we can write
$$
\overline{\{p\}}\setminus\{p\}~=~\bigcup_{q\in \Lambda\subseteq \overline{\{p\}}\setminus\{p\}}\ \overline{\{q\}},
$$
where the subset $\Lambda$ of $\overline{\{p\}}\setminus\{p\}$ has cardinality strictly less than $\kappa$.
\end{definition}

We can use this separable condition to describe when a point is locally closed in a Zariski space. The following result is probably well-known in algebraic geometry. We provide a proof in order to show that a point is locally closed in a Zariski space is equivalent to the fact that it is $\aleph_0$-separable.

\begin{lemma}\label{lem:Zlocallyclosed}
Let $X$ be a Zariski space. Then a point $p\in X$ is locally closed if and only if it is $\aleph_0$-separable.
\end{lemma}
\begin{proof}
It is standard and easy to see that any closed subset of a noetherian topological space is a finite union of irreducible closed subsets; see \cite[chap. II \S4.2]{Bo89} and \cite[Exercise 3.17]{Hart77}. Furthermore since $X$ is Zariski, every irreducible closed subset is the closure of some generic point. Then it is easy to deduce the equivalence of the statement from Definition-Lemma \ref{dl:clopen}(iii).
\end{proof}

Our next observation is that whether a point is $\kappa$-separable or not can be detected locally in a Zariski space.

\begin{lemma}\label{lem:localpropertysep}
Let $X$ be a Zariski space, and $\kappa$ be an infinite cardinal number. A point $p\in X$ is $\kappa$-separable in $X$ if and only if  $p$ is $\kappa$-separable in some locally closed subset $Y$ containing $p$, where $Y$ is endowed with the subspace topology.
\end{lemma}
\begin{proof}
First, it is routine to check that any locally closed subset $Y$ in $X$ with the subspace topology is again Zariski. Suppose $p$ is $\kappa$-separable in $X$. Then there is a subset $\Lambda \subseteq \overline{\{p\}}\setminus \{p\}$ of cardinality strictly less than $\kappa$ such that
$$
\overline{\{p\}}\setminus\{p\}~=~\bigcup_{q\in \Lambda}\, \overline{\{q\}}.
$$
Hence for any locally closed subset $Y$ containing $p$ we have
$$
\left(\overline{\{p\}}\cap Y\right)\setminus\{p\}~=~\bigcup_{q\in \Lambda'}\, \left(\overline{\{q\}}\cap Y\right),
$$
where $\Lambda'\subseteq \Lambda$ consists of those $q\in \Lambda$ such that $\overline{\{q\}}\cap Y\neq \emptyset$. It is easy to see that $\overline{\{q\}}\cap Y$ is the closure of the point $q\in Y$. So the above equality shows that $p$ is $\kappa$-separable in $Y$.

Conversely, suppose $p$ is $\kappa$-separable in some locally  closed subset $Y$ containing $p$. So
$$
\left(\overline{\{p\}}\cap Y\right)\setminus\{p\}~=~\bigcup_{q\in \Lambda}\, \left(\overline{\{q\}}\cap Y\right),
$$
for some subset $\Lambda\subseteq (\overline{\{p\}}\cap Y)\setminus\{p\}$ of cardinality strictly less than $\kappa$. By Lemma \ref{dl:clopen}(i), we can write $Y=U\cap V$ with $U$ open and $V$ closed. So we get $\overline{\{q\}}\cap Y=\overline{\{q\}}\cap U$ for any $q\in Y$. Decompose the closed subset $\overline{\{p\}}\setminus U$ into
$$\overline{\{p\}}\setminus U~=~Y_1\cup \dots \cup Y_n~=~\overline{\{q_1\}}\cup\cdots \cup \overline{\{q_n\}},$$
where $q_1,\ldots,q_n$ are generic points of the closed irreducible components $Y_1,\ldots,Y_n$.  Then we have
\begin{align*}
\overline{\{p\}}\setminus\{p\}~&=~\left(\left(\overline{\{p\}}\cap U\right)\setminus\{p\}\right)\bigcup \left(\overline{\{p\}}\setminus U\right)\\
~&=~\left(\bigcup_{q\in \Lambda} \left(\overline{\{q\}}\cap U\right)\right)\bigcup \left(\bigcup_{1\le i\le n} \overline{\{q_i\}}\right)
~=~\bigcup_{q\in \Lambda'} \overline{\{q\}},
\end{align*}
where $\Lambda'=\Lambda\cup\{q_1,\ldots,q_n\}$ has cardinality strictly less than $\kappa$ since $|\Lambda|<\kappa$ and $\kappa$ is infinite.
\end{proof}

It is important to point out that the criterion of $\kappa$-separability for a Zariski space $X$ is uniquely determined by the poset $(X,\le)$, where the partial order $\le$ is defined by the relation $x\le y$ if and only if $y\in \overline{\{x\}}$ for any $x,y\in X$. Recall for an element $s$ in a poset $(S,\le)$, a {\it cover} of $s$ is a minimal element in the subset $\{t\in S\,|\, s\lneqq t\}$. 

\begin{definition}
Let $(S,\le)$ be a poset, and $\kappa$ be a cardinal number. An element $s\in S$ is said to be {\it $\kappa$-separable} if the set of all covers of $s$ has cardinality strictly less than $\kappa$. 
\end{definition}

\begin{lemma}\label{lem:equisep}
Let $X$ be a Zariski space and consider $(X,\le)$ as a poset where $x\le y$ if $y\in \overline{\{x\}}$ for any $x,y\in X$. Suppose $\kappa$ is a cardinal number. For any point $p\in X$, the following are equivalent:
\begin{enumerate}
\item $p$ is $\kappa$-separable in the Zariski space $X$;
\item $p$ is $\kappa$-separable in the poset $(X,\le)$;
\item there exists a subset $\Lambda$ of $\{x\in X\,|\, x\gneqq p\}$ with $|\Lambda|<\kappa$ satisfying for any $y\gneqq p$ there is some $x\in \Lambda$ such that $y\ge x$.
\end{enumerate}
\end{lemma}
\begin{proof}
(ii)$\Leftrightarrow$(iii) Since $X$ is Zariski, it satisfies the descending chain condition on closed subsets. By definition, it implies that for any $r\gneqq p$ in the poset $(X,\le)$ there could be only finitely many elements between them. So for any $p$ in the poset $(X,\le)$, having a subset $\Lambda$ in $X$ satisfying the properties: $|\Lambda|<\kappa$ and for any $r\gneqq p$ there is some $q\in \Lambda$ such that $r\ge q$, is the same as the set of all covers of $p$ having cardinality strictly less than $\kappa$. 

(i)$\Leftrightarrow$(iii) Now suppose we can write
$$
\overline{\{p\}}\setminus\{p\}~=~\bigcup_{q\in \Lambda \subseteq \overline{\{p\}}\setminus\{p\}}\, \overline{\{q\}}.
$$
This is equivalent to say that for any $r\in \overline{\{p\}}\setminus\{p\}$ there is some $q\in \Lambda$ such that $r\in \overline{\{q\}}$. By definition in the poset $(X,\le)$, this is equivalent to the fact that for any $r\gneqq p$ there is some $q\in \Lambda$ such that $r\ge q$. So our result follows from direct translation.   
\end{proof}

\begin{remark}\label{rem:ken}
In an arbitrary poset $(S,\le)$ for any element $s\in S$ that is $\kappa$-separable there may not exists a subset $\Lambda$ of $S$ with $|\Lambda|<\kappa$ satisfying for any $r\gneqq s$ there is some $x\in \Lambda$ such that $r\ge x$. For example, consider the interval $[1,\gamma]$ with reversed order where $\gamma$ is the first uncountable ordinal, then $\gamma$ has no covers, it is automatically $\aleph_0$-separable but there does not exist a countable subset $\Lambda$ in $[1,\gamma]$ satisfying for any $a\gneqq \gamma$ there is some $b\in \Lambda$ such that $a\ge b$.
\end{remark}

\section{Commutative differential algebras}\label{S:D}
In this section, we follow \cite{Go06} to discuss Dixmier-Moeglin equivalence in the context of a commutative algebra together with a set of derivations.

A {\it commutative differential $k$-algebra} is a pair $(R,\Delta)$, where $R$ is a commutative $k$-algebra and $\Delta\subseteq {\rm Der}_kR$ is a subset of $k$-linear derivations of $R$. Throughout, let $(R,\Delta)$ be a commutative differential $k$-algebra.

{\it A $\Delta$-ideal} of $R$ is any ideal $I$ of $R$ satisfying $\delta(I)\subseteq I$ for any $\delta\in \Delta$. A prime ideal which is also a $\Delta$-ideal is called a {\it prime $\Delta$-ideal}. There is another notion of {\it $\Delta$-prime ideal} referring to any $\Delta$-ideal that is prime among all $\Delta$-ideals. These two concepts coincide when $R$ is noetherian; see reference \cite[Lemma 1.1(d)]{Go06}. We use $\Delta$-${\rm spec}\, R$ to denote the set of all $\Delta$-prime ideals in $R$, which is a Zariski space with the induced Zariski topology via the obvious embedding $\Delta$-${\rm spec}\, R\subseteq {\rm spec}\, R$. Since derivations are extendable to localizations, any localization of $R$ is still a commutative differential $k$-algebra.

For any ideal $J$ of $R$, the {\it $\Delta$-core} of $J$ is the largest $\Delta$-ideal contained in $J$, denoted by
$$\left(J:\Delta\right)~:=~\left\{r\in R\,|\, \delta_1\cdots \delta_n(r)\in J,\ \forall\ \delta_1,\ldots,\delta_n\in \Delta,\, n\ge 0\right\}.$$
In particular, $\Delta$-cores of maximal ideals of $R$ are called {\it $\Delta$-primitive ideals}. We use $\Delta$-${\rm prim}\, R$ to denote the set of all $\Delta$-primitive ideals in $R$.

The {\it $\Delta$-center} of $R$ is the subalgebra
$$Z_\Delta(R)~:=~\left\{r\in R\, |\,  \delta(r)=0,\ \forall\ \delta\in \Delta\right\}.$$
For any $\Delta$-prime ideal $P$ of $R$, all derivations in $\Delta$ can pass to $R/P$ and then extend uniquely to $\Frac R/P$. We say that $P$ is {\it $\Delta$-rational} if the field $Z_\Delta(\Frac\, R/P)$ is algebraic over the base field $k$.

By analogy with the Dixmier-Moeglin equivalence for algebras, we say that $R$ satisfies the {\it $\Delta$-Dixmier-Moeglin equivalence} provided the following equivalences hold for any $P\in \Delta$-${\rm spec}\, R$
$$P~{\bf locally~closed~in\,\, \Delta}\text{-}{\rm spec}\, R \Longleftrightarrow P~{\bf \Delta\text{-}primitive}\Longleftrightarrow P~{\bf \Delta\text{-}rational}.$$

We recall some basic properties about $\Delta$-${\rm spec}\, R$. In properties below, (i) is Dixmier's result \cite[Lemma 3.3.2]{Dix77e} and (iv) is proved in \cite[Lemma 1.1(e)]{Go06} when $R$ is affine over $k$ where we prove the same result for those $R$ that are noetherian over large base fields.

\begin{lemma}\label{DeltaP}
Let $(R,\Delta)$ be a commutative differential $k$-algebra.
\begin{enumerate}
\item The $\Delta$-core of any prime ideal of $R$ is prime. In particular, every $\Delta$-primitive ideal of $R$ is prime.
\item There is a natural homeomorphism between $\Delta$-${\rm spec}\, RS^{-1}$ for the localization of $R$ at any multiplicative closed set $S$ in $R$ and the subset of $\Delta$-${\rm spec}\, R$ containing those $\Delta$-prime ideals that do not intersect with $S$ equipped with the subspace topology.
\item[(iii)]  If $R$ is noetherian, then the set of all minimal $\Delta$-prime ideals containing any $\Delta$-ideal is finite. In particular, every $\Delta$-semiprime ideal is the intersection of finitely many  $\Delta$-prime ideals.
\item[(iv)] If $R$ is noetherian and $\dim_k R<|k|$, then every  $\Delta$-prime ideal is an intersection of $\Delta$-primitive ideals.
\end{enumerate}
\end{lemma}
\begin{proof}
(i) \cite[Lemma 3.3.2]{Dix77e}.

(ii) It follows from the well-known one-to-one correspondence for prime ideals in localizations.

(iii) One can see from (i) that every prime ideal minimal over a $\Delta$-ideal is a $\Delta$-prime ideal of $R$. So it follows from Noether's theorem.

(iv) By \cite[Proposition II.7.12]{BrGo}, we know $R$ is a Jacobson ring. So every prime ideal of $R$ is the intersection of maximal ideals containing it. Then for any $\Delta$-prime ideal $J$ of $R$ we have
\begin{align*}
J~&=~\left(J:\Delta\right)~=~\bigcap_{J\subseteq P\in {\rm max}\,R} P~=~\left(\left(\bigcap_{J\subseteq P\in {\rm max}\,R} P\right):\Delta\right)\\
~&=~\bigcap_{J\subseteq P\in {\rm max}\,R} \left(P:\Delta\right)~=~\bigcap_{J\subseteq Q\in \Delta\text{-}{\rm prim}\,R} Q.
\end{align*}
\end{proof}

\begin{remark}
As a consequence of Lemma \ref{DeltaP}(i), we see that
$$\Delta\text{-}{\rm prim}\, R~\subseteq~\Delta\text{-}{\rm spec}\, R~\subseteq~{\rm spec}\, R,$$
which can be considered as inclusions of Zariski spaces.
\end{remark}

The following result is well-known for the prime spectrum of any associative algebra, e.g., \cite[Lemma II.7.11]{BrGo}, we adapt the version for $\Delta$-prime spectrums here.

\begin{lemma}\label{lem:equilocallyclosed}
Let $(R,\Delta)$ be a commutative differential $k$-algebra. For any $\Delta$-prime ideal $P$ of $R$, the following are equivalent:
\begin{enumerate}
\item $P$ is locally closed in $\Delta$-${\rm spec}\, R$;
\item There exists $f\neq 0$ in the quotient algebra $R/P$ such that the localization $(R/P)_f$ has no proper nontrivial $\Delta$-prime ideals;
\item The intersection of all $\Delta$-prime ideals properly containing $P$ is a $\Delta$-ideal properly containing $P$;
\item $P$ is $\aleph_0$-separable in $\Delta$-${\rm spec}\, R$ assuming that $R$ is noetherian.
\end{enumerate}
\end{lemma}
\begin{proof}
(i)$\Rightarrow$(ii) By Definition-Lemma \ref{dl:clopen}(iii), we can write
$$\{P\}=\{Q\in \Delta\text{-}{\rm spec}\, R\,|\, Q \supseteq P\}\setminus \{Q\in \Delta\text{-}{\rm spec}\, R\,|\, Q\supseteq I\}$$
for some $\Delta$-ideal $I$ strictly containing $P$. Pick any $f\in I\setminus P$ and still write $f$ as its image in $I/P$. Since every $\Delta$-prime ideal $Q$ strictly containing $P$ must contain $I$, we have $f\in I\subseteq Q$. So by Lemma \ref{DeltaP}(ii), $(R/P)_f$ has no proper nontrivial $\Delta$-prime ideals.

(ii)$\Rightarrow$(iii) Choose any preimage of $f$ in $R$, which we still denote by $f$ such that $f\in R\setminus P$. Again by Lemma \ref{DeltaP}(ii), every $\Delta$-prime ideal strictly containing $P$ has to contain $f$. Thus $f$ is contained in the intersection of all $\Delta$-prime ideals properly containing $P$, which then is a $\Delta$-ideal properly containing $P$.

(iii)$\Rightarrow$(iv) Let $I$ be the intersection of all $\Delta$-prime ideals properly containing $P$, which is a $\Delta$-ideal properly containing $P$. Since $R$ is noetherian, by Lemma \ref{DeltaP}(iii) the set of all minimal $\Delta$-prime ideals containing $I$ is finite. So $P$ is $\aleph_0$-separable in $\Delta$-${\rm spec}\, R$ by Lemma \ref{lem:equisep}.

(iv)$\Rightarrow$(i) It is Lemma \ref{lem:Zlocallyclosed}.
\end{proof}

The next result was first observed in the proof of \cite[Proposition 1.2]{Go06} for $\Delta$-primitive ideals. The proof given there extends verbatim to general $\Delta$-prime ideals.

\begin{lemma}\label{lem:ZEmbed}
Let $(R,\Delta)$ be a commutative differential $k$-algebra. For any prime ideal $P$ of $R$, there exists a field extension
$$\phi: Z_\Delta\left(\Frac\, R/(P:\Delta)\right)\longrightarrow \Frac R/P$$
over the base field $k$.
\end{lemma}

The most difficult case in $\Delta$-Dixmier-Moeglin equivalence is to show that every $\Delta$-rational ideal of $R$ is locally closed in $\Delta\text{-}{\rm spec}\, R$ due to the following result, which was first given in the complex affine Poisson case by Oh \cite[Propositions 1.7, 1.1]{ Oh99} and later proved for any commutative affine differential $k$-algebra by Goodearl \cite[Proposition 1.2]{Go06}. We show the same result holds for any commutative noetherian differential $k$-algebra over large base field.

\begin{proposition}\label{weakPDME}
Let $(R,\Delta)$ be a commutative differential $k$-algebra. Suppose $R$ is noetherian and $\dim_k R<|k|$. Then we have the following implications for any $P\in \Delta$-${\rm spec}\, R$:
$$P~\text{locally~closed~in}\,\, \Delta\text{-}{\rm spec}\, R\Longrightarrow P~{\Delta\text{-}primitive}\Longrightarrow P~{\Delta\text{-}rational}.$$
\end{proposition}
\begin{proof}
Suppose $P$ is locally closed in $\Delta$-${\rm spec}\, R$. By Lemma \ref{DeltaP}(iv), we have $P=\bigcap_{P\subseteq P_i\in \Delta\text{-}{\rm prim}\,R} P_i$. Now we must have some $P_i=P$ since Lemma \ref{lem:equilocallyclosed}(iii) implies that
$$P~\subsetneq~\bigcap_{P\subsetneq P_i\in \Delta\text{-}{\rm spec}\,R} P_i~\subseteq~\bigcap_{P\subsetneq P_i\in \Delta\text{-}{\rm prim}\,R} P_i.$$
This shows that $P$ is $\Delta$-primitive.

Next let $P$ be $\Delta$-primitive and we can write $P=(M:\Delta)$ for some $M\in {\rm max}\, R$. By Lemma \ref{lem:ZEmbed}, we have a field extension $Z_\Delta(\Frac R/P)\to \Frac(R/M)$ over $k$.  Since $\dim_k R<|k|$, we have that $R$ satisfies the Nullstellensatz \cite[Proposition II.7.16]{BrGo}. So $Z_\Delta(\Frac R/P)\subseteq \Frac(R/M)$ is algebraic over $k$ by Definition \cite[Definition II.7.14]{BrGo}. So $P$ is $\Delta$-rational and it completes the proof.
\end{proof}

\section{Main results}\label{S:M}
We need some more lemmas before we can tackle our main theorems. The next lemma is originated from the well-known results in ring theory concerning ideals in rings obtained by extending scalars in centrally closed algebras, e.g., \cite[Lemma 2.1]{BWY19}. We establish the result in the setting of any commutative differential $k$-algebra by following the proof of the latter result.

\begin{lemma}\label{lem1}
Let $(R,\Delta)$ be a commutative differential $k$-algebra. Suppose $K/k$ is any field extension and both $R$ and $(\Frac R)\otimes_kK$ are integral domains. If the zero ideal of $R$ is $\Delta$-rational, then every nonzero $\Delta$-ideal of $R\otimes_kK$ contains an element of the form $r\otimes 1$  for some $0 \neq r \in R$.
\end{lemma}
\begin{proof}
Let $I$ be a nonzero $\Delta$-ideal of $R\otimes_k K$ and choose a nonzero element $x\in I$ such that $x=\sum_{i=1}^d r_i\otimes \lambda_i$ with $r_i\in R$, $\lambda_i\in K$ and $d$ minimal.  We claim that $d=1$.  If not, suppose that $d>1$. Clearly all $r_1,\ldots,r_d$ are nonzero in $R$. So in $(\Frac R)\otimes_k K$, we may write $x=(r_1\otimes 1)y$, where $y=\sum_{i=1}^d z_i\otimes \lambda_i$ with $z_1=r_1/r_1,\ldots ,z_d=r_d/r_1\in \Frac R$. Set $F:=k(z_1,\dots,z_d)$ to be the extension of $k$ in $\Frac R$ generated by $z_1,\ldots ,z_d$. We claim that $F$ is algebraic over $k$. We note, for any $\delta\in \Delta$ which can be extended to a $K$-linear derivation of $R\otimes_kK$, the following element
$$
\delta(x)(r_1\otimes 1)-x(\delta(r_1)\otimes 1)~=~\sum_{i=2}^d(\delta(r_i)r_1-r_i\delta(r_1))\otimes \lambda_i
$$
is again in $I$.  By minimality of $d$, we have $\lambda_2,\ldots ,\lambda_d$ are $k$-linearly independent and $\delta(r_i)r_1=r_i\delta(r_1)$ for all $2\le i\le d$. In terms of $\Frac R$, the above equality can be read as $\delta(r_i/r_1)=0$ for all $\delta\in\Delta$ or
equivalently $z_i\in Z_\Delta(\Frac R)$ for all $1\le i\le d$. Since $(0)$ is rational, we have that all $z_i$'s are algebraic over $k$. This proves the claim.

Now since $[F:k]<\infty$, we see that $[F\otimes_kK: K]<\infty$ and $y\in F\otimes_k K$ is algebraic over $K=k\otimes_k K$.  In particular, there is a non-trivial relation $y^m (1\otimes c_m)+ y^{m-1} (1\otimes c_{m-1}) + \cdots + (1\otimes c_0)=0$ for some integer $m\ge 1$ and $c_i\in K$.  Furthermore, we may assume $c_0$ is nonzero since $F\otimes_k K\subseteq (\Frac R)\otimes_kK$ is an integral domain. Without loss of generality, set $c_0=1$.  Then by construction
\begin{align*}
r_1^m~&=~\sum_{j=0}^m (r_1\otimes 1)^{m-j} x^j (1\otimes c_j)-\sum_{j=1}^m (r_1\otimes 1)^{m-j} x^j (1\otimes c_j)\\
~&=~-\sum_{j=1}^m (r_1\otimes 1)^{m-j} x^j (1\otimes c_j)\in I.
\end{align*}
Since $r_1$ is nonzero in the integral domain $R$, we have $r_1^m$ is nonzero and the result follows.
\end{proof}

\begin{remark} We note that this result need not hold if $(\Frac R)\otimes_k K$ is not an integral domain.  For example, if $R=K=\mathbb{Q}(\sqrt{2})$ and $k=\mathbb{Q}$.  Then $R\otimes_k K$ is not an integral domain and hence there is some nonzero prime ideal $I$ of $R\otimes_k K$.  But $R$ is a field, so $I\cap (R\otimes 1)$ is necessarily zero. 
\end{remark}

The following key lemma reveals that, for a commutative noetherian differential $k$-algebra $(R,\Delta)$, the rationality of $\Delta$-prime ideals can be purely captured in terms of the poset $(\Delta\text{-}{\rm spec}\,R, \subseteq)$ whenever both ${\rm dim}_k R$ and $|\Delta|<|k|$. The proof we give below follows closely the proof of \cite[Lemma 2.3]{BWY19}.

\begin{lemma} Let $(R,\Delta)$ be a commutative differential $k$-algebra. Suppose $R$ is noetherian and both ${\rm dim}_k R$ and $|\Delta|<|k|$. Then any $\Delta$-prime ideal  is $\Delta$-rational if and only if it is $|k|$-separable in $\Delta$-${\rm spec}\ R$.
\label{card}
\end{lemma}
\begin{proof}
Let $P$ be a $\Delta$-prime ideal of $R$. By replacing $R$ with $R/P$, we may assume that $P=(0)$ and $R$ is a noetherian integral domain. Notice that if $|k|\le \aleph_0$ then ${\rm dim}_k R<\infty$ and so $R$ is just a finite field extension of $k$. Then $(0)$ is $\Delta$-rational, maximal ideal and the claim is vacuously true in this case.  Thus we assume henceforth that $k$ is uncountable.

Let $\mathcal{B}=\{r_{\alpha}\colon \alpha\in \Lambda\}$ be a $k$-basis for $R$, where $\Lambda$ is an index set with $|\Lambda|=\dim_k R<|k|$.  Then for any $\alpha,\beta\in \Lambda$ and $\delta\in \Delta$, we can write
$$r_{\alpha}r_{\beta}~=~\sum_{\gamma\in \Lambda} c_{\alpha,\beta}^{\gamma}\, r_{\gamma}\quad \text{and}\quad \delta(r_\alpha)~=~\sum_{\gamma\in \Lambda}d_{\delta,\alpha}^{\gamma}\, r_\gamma,$$
where the coefficients $c_{\alpha,\beta}^{\gamma}, d_{\delta, \alpha}^{\gamma}\in k$ are zero for all but finitely many $\gamma\in \Lambda$.  Consider the subset
$$T~:=~\left\{c_{\alpha,\beta}^{\gamma}, d_{\delta,\alpha}^{\gamma}\,\Big |\, \forall\ \alpha,\beta,\gamma\in \Lambda, \delta\in \Delta\right\}$$
of $k$. Since for any $(\alpha,\beta,\delta)\in \Lambda\times \Lambda\times \Delta$ there are only finitely many values of $\gamma$ for which $c_{\alpha,\beta}^{\gamma}$ and $d_{\delta,\alpha}^{\gamma}$ are nonzero, we can construct an injection from
$T\hookrightarrow \Lambda\times \Lambda\times \Delta \times \mathbb{N}$. In particular, $|T|\le |\Lambda|^2 |\Delta||\mathbb{N}| < |k|$, since $k$ is uncountable and $|\Lambda|,|\Delta|<|k|$.

Now let $k_0$ be the prime subfield of $k$ and let $F=k_0(T)$ be the extension of $k_0$ generated by $T$.  Set $R_0 :=\sum_{\alpha\in \Lambda} F r_{\alpha}$ to be the $F$-subspace of $R$. We claim that $R_0$ is a commutative  differential $F$-subalgebra of $R$ that is a noetherian integral domain satisfying $R_0\otimes_Fk\cong R$ and $|\Frac R_0|<|k|$.
By construction, $R_0$ is closed under multiplication of $R$ and is invariant under the actions of $\Delta$. So $R_0$ is a commutative differential  $F$-subalgebra of $R$ satisfying $R_0\otimes_Fk\cong R$. Now since $R$ is a free $R_0$-module, we have an inclusion-preserving injection from the set of left $R_0$-modules to the set of left $R$-modules. In particular, $R_0$ is noetherian since $R$ is. Clearly $R_0\subset R$ is an integral domain since $R$ is. Moreover, since ${\rm dim}_F R_0 < |k|$, $|F|\le \aleph_0|T|<|k|$, and $k$ is uncountable, we have $|R_0|<|k|$. Since every element of $\Frac R_0$ can be expressed in the form $sr^{-1}$ with $s,r\in R_0$, we have that $|\Frac R_0|<|k|$. This proves the claim.

We first show that $|k|$-separability implies $\Delta$-rationality. Suppose that $(0)$ is not $\Delta$-rational. Then there is some $z=ab^{-1}\in Z_\Delta(\Frac R)$ that is not algebraic over $k$ for some $a,b\in R$. Since $z$ is not in $k$, we have $a$ and $b$ are $k$-linearly independent and we may assume without loss of generality that $a,b\in \mathcal{B}\subset R_0$ in the above construction. So $z=ab^{-1}\in \Frac R_0$. Now for any $\lambda\in k$, consider the element $z_{\lambda}  := z\otimes 1 - 1\otimes \lambda$ in $(\Frac R_0)\otimes_F k$ and the subset
$$U~:=~\left\{\lambda\in k\,|\, z_{\lambda}~{\rm is~a~unit~in~}(\Frac R_0)\otimes_F k\right\}$$
of $k$. We show that $|U|<|k|$ by the Amitsur trick. If not, suppose $|U|=|k|$. Since $(\Frac R_0)\otimes_F k$ has dimension strictly less than $|k|$, there is necessarily a (finite) $k$-dependence of $z_{\lambda}^{-1}$ with $\lambda\in U$; after clearing denominators in this dependence, we get that $z$ is algebraic over $k$, which is a contradiction.

We see, as a localization of $R$, $(\Frac R_0)\otimes_F k$ is also a noetherian integral domain. Let $\lambda\in k\setminus U$.
By Krull's principal ideal theorem, we have $(z_{\lambda})$ is contained in some height one prime ideal $P_{\lambda}$ of $(\Frac R_0)\otimes_F k$. Moreover, the $\Delta$-core of $P_\lambda$ is nonzero since it contains $z_\lambda\in Z_\Delta(\Frac R)$ and is again prime by Lemma \ref{DeltaP}(i). This forces $P_\lambda=(P_\lambda:\Delta)$ to be a height one $\Delta$-prime ideal of $(\Frac R_0)\otimes_K k$. By Lemma \ref{DeltaP}(ii), $Q_{\lambda}:=P_{\lambda}\cap R$ is a height one $\Delta$-prime ideal of $R$ . We notice that $Q_\lambda\neq Q_{\lambda'}$ if $\lambda\neq \lambda'$ in $k\setminus U$. Otherwise, it would imply $0\neq \lambda'-\lambda=z_\lambda-z_\lambda'\in Q_\lambda=Q_{\lambda'}$ and we get a contradiction.

Now let us assume that $(0)$ is $|k|$-separable in $\Delta$-${\rm spec}\, R$, namely there exists an index set $\mathscr S$ with $|\mathscr S|<|k|$ and a set of nonzero $\Delta$-prime ideals $\{P_s\,|\, s\in \mathscr S\}$ such that every nonzero $\Delta$-prime ideal of $R$ contains some $P_s$.  Then since $|\mathscr S|<|k|$ and $|k\setminus U|=|k|$, there is some $s\in \mathscr S$ such that $P_s\subseteq Q_{\lambda}$ for infinitely many $\lambda\in k\setminus U$. Since $Q_{\lambda}$'s are height one primes of $R$, we have that $0\subsetneq P_s\subseteq Q_\lambda$ implies that $P_s=Q_\lambda$. But $Q_\lambda\neq Q_{\lambda'}$ if $\lambda\neq \lambda'$, so we get a contradiction. Thus we have shown one direction of the lemma.

For the other direction such that $\Delta$-rationality implies $|k|$-separability, we suppose that $(0)$ is $\Delta$-rational.  Retaining the above notations,  for any $0\neq a\in R_0$, let $\mathscr S_a$ denote the set of $\Delta$-prime ideals of $R$ containing $a\otimes 1$. By Lemma \ref{lem1}, since $(0)$ is rational and $R_0$ and $(\Frac R_0)\otimes_F k$ are integral domains, every nonzero $\Delta$-ideal of $R_0\otimes_F k\cong R$ contains an element of the form $a\otimes 1$ with some nonzero $a$ in $R_0$. Then
$$\Delta\text{-}{\rm spec}\, R\setminus \{(0)\}~=~\bigcup_{a\in R_0\setminus \{0\}} \mathscr S_a.$$
Now for each nonzero $a$ in $R_0$, let $P_a = \bigcap_{P\in \mathscr S_a} P$.  Then $P_a$ is a nonzero $\Delta$-semiprime ideal of $R$ since $a\otimes 1\in P_a$.  Thus by Lemma \ref{DeltaP}(iii), $P_a$ is a finite intersection of nonzero $\Delta$-prime ideals, say $P=P_{a,1}\cap \cdots \cap P_{a,n_a}$ and every nonzero $\Delta$-prime ideal containing $P_a$ contains some $P_{a,i}$. Finally, we let
$$\mathscr S~=~\left\{P_{a,i}\,|\, a\in R_0\setminus \{0\},\, 1\le i\le n_a\right\}.$$  Then $\mathscr S$ is a set of nonzero $\Delta$-prime ideals of $R$ of cardinality at most $|R_0|\times |\mathbb{N}| < |k|$ and by construction, every nonzero $\Delta$-prime ideal of $R$ contains some $\Delta$-prime ideal from $\mathscr S$.  So $(0)$ is $|k|$-separable in $\Delta$-${\rm spec}\,R$ by Lemma \ref{lem:equisep}. This completes the proof.
\end{proof}

\begin{remark}
For a commutative differential $k$-algebra $(R,\Delta)$, we can always replace $\Delta$ by a basis of the $k$-linear span of $\Delta$ in $\Der_k\, R$. So the condition $|\Delta|<|k|$ can be replaced by $\dim_k\, {\rm span}_k(\Delta)<|k|$. Now if $\aleph_0\le \dim_k\, R<|k|$ and $\Der_k R$ is a finitely generated $R$-module, we have $\dim_k\, {\rm span}_k(\Delta)\le \dim_k (\Der_k\, R)\le \dim_k\, R< |k|$. 
\end{remark}

Now we are able to give our main theorems. The first main result is a generalization of \cite[Theorem 7.1]{BLLM17}, which can be thought as a weak version of the $\Delta$-Dixmier-Moeglin equivalence. In particular for the Poisson Dixmier-Moeglin equivalence, the third topological condition below was replaced by the finiteness condition of Poisson prime ideals of height exactly increased by one in \cite[Theorem 7.1(3)]{BLLM17}.

\begin{theorem}\label{thm:rationalprimitive}
Let $(R,\Delta)$ be a commutative differential $k$-algebra. Suppose $R$ is noetherian and both ${\rm dim}_k R$ and $|\Delta|<|k|$. For a $\Delta$-prime ideal $P$ of $R$, the following are equivalent:
\begin{enumerate}
\item $P$ is $\Delta$-primitive;
\item $P$ is $\Delta$-rational;
\item $P$ is $|k|$-separable in $\Delta$-${\rm spec}\ R$.
\end{enumerate}
\end{theorem}
\begin{proof}
(i)$\Rightarrow$(ii) is Proposition \ref{weakPDME} and (ii)$\Leftrightarrow$(iii) is Lemma \ref{card}. So it suffices to show (iii)$\Rightarrow$(i). We may assume that $k$ is uncountable and $P=(0)$ where $R$ becomes a noetherian integral domain.  By definition, there is a subset $\{Q_s\in \Delta\text{-}{\rm spec}\, R\setminus\{(0)\}\,|\, s\in \Lambda\}$, where $\Lambda$ is an index set with $|\Lambda|<|k|$, such that every nonzero $\Delta$-prime ideal of $R$ must contain some $Q_s$. For each $s\in \Lambda$, choose some nonzero $f_s\in Q_s$. Denote by $T$ the multiplicative closed subset of $R$ generated by all $f_s$, and $B:=RT^{-1}$ the localization of $R$ at $T$. It is clear that $|T|\le \aleph_0|\Lambda|<|k|$ since $|k|$ is uncountable. So $B$ is noetherian with $\dim_k B<|k|$. Then $B$ satisfies the Nullstellensatz \cite[Proposition II.7.16]{BrGo} such that for any maximal ideal $J$ of $B$, $B/J$ is algebraic over $k$. Let $I:=J\cap R$ for any maximal ideal $J$ of $B$. Since $R/I$ embeds into $B/J$, we can see that $R/I$ too is an algebraic extension of $k$ and hence $I$ is maximal in $R$. By the construction of $T$, we have $I=J\cap R$ does not intersect $T$, which implies that $I$ does not contain any $Q_s$. But then we must have $(0)=(I:\Delta)$ is $\Delta$-primitive since otherwise $f_s\in Q_s\subseteq (I:\Delta)\subseteq I$ for some $Q_s$ for $(I:\Delta)$ is $\Delta$-prime by Lemma \ref{DeltaP}(i). This completes the proof.
\end{proof}

Our next main result reveals that, for any commutative noetherian differential $k$-algebra $(R,\Delta)$ whenever both $\dim_k R$ and $|\Delta|<|k|$, the Zariski topology of $\Delta\text{-}{\rm spec}\, R$ can detect the $\Delta$-Dixmier-Moeglin equivalence.

\begin{theorem}
\label{thm:main1}
Let $(R,\Delta)$ be a commutative differential $k$-algebra. Suppose that $R$ is noetherian and both ${\rm dim}_k R$ and $|\Delta|<|k|$. Then the following are equivalent:
\begin{enumerate}
\item $R$ satisfies the $\Delta$-Dixmier-Moeglin equivalence;
\item Every $\Delta$-rational ideal of $R$ is locally closed in $\Delta$-${\rm spec}\, R$;
\item Every $\Delta$-primitive ideal of $R$ is locally closed in $\Delta$-${\rm spec}\, R$;
\item In the Zariski space $\Delta$-${\rm spec}\, R$, that a $\Delta$-prime ideal $P$ is $|k|$-separable implies that it is $\aleph_0$-separable;
\item In the poset $(\Delta$-${\rm spec}\, R,\subseteq)$, that a $\Delta$-prime ideal $P$ is $|k|$-separable implies that it is $\aleph_0$-separable.
\end{enumerate}
\end{theorem}
\begin{proof}
The above equivalences are direct consequences of Proposition \ref{weakPDME}, Theorem \ref{thm:rationalprimitive}, and  Lemma \ref{lem:equisep}. Note that for $\Delta$-${\rm spec}\, R$, the partial order $\leq$ of Lemma \ref{lem:equisep} amounts to inclusion of $\Delta$-prime ideals.
\end{proof}

Now we give some consequences of our main results.

Brown and Gordon showed the Poisson Dixmier-Moeglin equivalence holds for any affine complex Poisson algebra with only finitely many Poisson primitive ideals \cite[Lemma 3.4]{BrGor03}. Our main theorems yield the similar result in terms of any commutative noetherian differential algebra over large base field with finitely many $\Delta$-primitive ideals.

\begin{corollary}
\label{cor:finitePrim}
Let $(R,\Delta)$ be a commutative differential $k$-algebra. Suppose that $R$ is noetherian with both ${\rm dim}_k R$ and $|\Delta|<|k|$. If every point in $\Delta$-${\rm spec}\, R$ is $\aleph_0$-separable, then $R$ satisfies the $\Delta$-Dixmier-Moeglin equivalence. In particular if $\Delta$-${\rm prim}\, R$ is finite, then $R$ satisfies the $\Delta$-Dixmier-Moeglin equivalence.
\end{corollary}
\begin{proof}
If any point in $\Delta$-${\rm spec}\, R$ is $\aleph_0$-separable, we know $R$ satisfies the $\Delta$-Dixmier-Moeglin equivalence by Theorem \ref{thm:main1}. In particular, suppose there are only finitely many $\Delta$-primitive ideals. We notice that every $\Delta$-prime ideal of $R$ is an intersection of $\Delta$-primitive ideals by Lemma \ref{DeltaP}(iv). Since there are only finitely many $\Delta$-primitive ideals, this implies that $\Delta$-${\rm spec}\, R$ is finite and hence every point in $\Delta$-${\rm spec}\, R$ is $\aleph_0$-separable. This completes the proof.
\end{proof}

Our next corollary is parallel to the similar result in \cite[Theorem 1.1]{BWY19}.

\begin{corollary}
\label{cor:inclusionPreserve}
Let $(R,\Delta_R)$ and $(S,\Delta_S)$ be two commutative differential $k$-algebras. Suppose that $R,S$ are noetherian with ${\rm dim}_k R, {\rm dim}_k S, |\Delta_R|, |\Delta_S|<|k|$. If there is an isomorphism between the two posets $(\Delta$-${\rm spec}\, R,\subseteq)$ and $(\Delta$-${\rm spec}\, S,\subseteq)$, then $R$ satisfies the $\Delta$-Dixmier-Moeglin equivalence if and only if $S$ satisfies the $\Delta$-Dixmier-Moeglin equivalence. In particular, if $\Delta$-${\rm spec}\, R$ and $\Delta$-${\rm spec}\, S$ are homeomorphic, then $R$ satisfies the $\Delta$-Dixmier-Moeglin equivalence if and only if $S$ does.
\end{corollary}
\begin{proof}
The result straight follows from Theorem \ref{thm:main1}(i)$\Leftrightarrow$(v) since the condition (v) is preserved under any inclusion-preserving bijection or homeomorphism between $\Delta$-${\rm spec}\, R$ and $\Delta$-${\rm spec}\, S$.
\end{proof}

\section{Stratification}\label{S:S}
In this section, we show that the $\Delta$-Dixmier-Moeglin equivalence is a local property for the Zariski space $\Delta$-${\rm spec}\, R$ while working over large base fields. Thus by considering proper stratification of the Zariski space $\Delta$-${\rm spec}\, R$, we can establish the $\Delta$-Dixmier-Moeglin equivalence by verifying it locally in each stratification.

\begin{theorem}
\label{thm:stratification}
Let $(R, \Delta)$ be a commutative noetherian differential $k$-algebra with both $\dim_k R$ and $|\Delta|<|k|$. Suppose that $\Delta$-${\rm spec}\, R=\bigcup X_i$ is a union of locally closed subsets with each $X_i$, when endowed with the subspace topology, homeomorphic to $\Delta$-${\rm spec}\, R_i$ for some noetherian commutative differential $k$-algebra $(R_i, \Delta_i)$ with both $\dim_k R_i$ and $|\Delta_i|<|k|$. Then $R$ satisfies the $\Delta$-Dixmier-Moeglin equivalence if and only if each $R_i$ satisfies the $\Delta_i$-Dixmier-Moeglin equivalence.
\end{theorem}
\begin{proof}
By Theorem \ref{thm:main1}, we know $R$ satisfies the $\Delta$-Dixmier-Moeglin equivalence if and only if that any point $P$ in $\Delta\text{-}{\rm spec}\, R$ is $|k|$-separable implies that it is $\aleph_0$-separable. Now by Lemma \ref{lem:localpropertysep}, we know that the $|k|$-separability (resp. $\aleph_0$-separability) can be determined in any locally closed subset of $\Delta\text{-}{\rm spec}\, R$. In particular, $P$ is $|k|$-separable (resp. $\aleph_0$-separable) if and only if $P$ is $|k|$-separable (resp. $\aleph_0$-separable) in any locally closed subset $X_i$ containing $P$. Thus the result follows.
\end{proof}

The $\Delta$-Dixmier-Moeglin equivalence has been established for any differential commutative $k$-algebra $R$ with suitable torus actions by Goodearl \cite{Go06}. In particular, in these cases it is shown that the corresponding $\Delta$-${\rm spec} R$ is a finite disjoint union of locally closed subsets $X_i$, where each $X_i$ is homeomorphic to the prime spectrum of some Laurent polynomial ring. We summarize the above illustration by asserting a local topological criterion when the topological space $\Delta$-${\rm spec} R$ is formed by glueing together finitely many noetherian affine schemes as locally closed subsets.

\begin{corollary}\label{cor:commutative}
Let $(R, \Delta)$ be a commutative noetherian differential $k$-algebra with both $\dim_k R$ and $|\Delta|<|k|$. Suppose that $\Delta$-${\rm spec}\, R=\bigcup X_i$ is a union of locally closed subsets with each $X_i$, when endowed with the subspace topology, homeomorphic to ${\rm spec}\, R_i$ for some noetherian commutative $k$-algebra $R_i$ with $\dim_k R_i<|k|$. Then $R$ satisfies the $\Delta$-Dixmier-Moeglin equivalence.
\end{corollary}
\begin{proof}
We notice that any commutative noetherian $k$-algebra $R_i$ with $\dim_k R_i<|k|$ satisfies Dixmier-Moeglin equivalence. Moreover, when $R_i$ is considered as a differential commutative $k$-algebra with $\Delta=\{0\}$, it automatically satisfies the $\Delta$-Dixmier-Moeglin equivalence. So the result follows from Theorem \ref{thm:stratification}.
\end{proof}

Next, we extend the symplectic core stratification on the maximal spectrum of a complex affine Poisson algebra discussed in \cite[\S 3.3]{BrGor03} to the $\Delta$-prime spectrum of a commutative differential $k$-algebra $(R,\Delta)$.

We define a relation $\sim$ on ${\rm spec}\, R$ by:
$$
P\sim Q~\Longleftrightarrow~(P:\Delta)=(Q:\Delta).
$$
We can easily check that $\sim$ is an equivalence relation on ${\rm spec}\, R$, where we denote the corresponding equivalence class of any $P\in {\rm spec}\, R$ by $\mathscr C_{\rm spec}(P)$, so that
$${\rm spec}\, R~=~ \bigsqcup_{Q\in \Delta\text{-}{\rm spec}\,R}\quad \mathscr C_{\rm spec}(Q)~=~\{P\in {\rm spec}\, R\,|\, (P:\Delta)=Q\}.$$
The set $\mathscr C_{\rm spec}(P)$ is called the {\it prime differential core} (of $P$). Moreover, the {\it (maximal) differential core} of $P$ is denoted by $\mathscr C(P):=\mathscr C_{\rm spec} (P)\cap {\rm max}\, R$ and we have
$${\rm max}\, R~=~ \bigsqcup_{Q\in \Delta\text{-}{\rm prim}\,R}\quad \mathscr C(Q)~=~\{P\in {\rm max}\, R\,|\, (P:\Delta)=Q\}.$$

Notice that there is a continuous retraction $\pi: {\rm spec}\, R\to \Delta\text{-}{\rm spec}\, R$ given by the $\Delta$-core such that $\pi(P)=(P:\Delta)$ for any $P\in {\rm spec}\, R$ and $\Delta\text{-}{\rm spec}\, R$ is a topological quotient of ${\rm spec}\, R$ via $\pi$ \cite[Theorem 1.3]{Go06}. Moreover, the restriction of $\pi$ on ${\rm max}\, R$ gives a continuous surjection denoted by $\pi_m: {\rm max}\, R\to \Delta\text{-}{\rm prim}\, R$ \cite[Theorem 1.5]{Go06}. In general, $\Delta\text{-}{\rm prim}\, R$ is not a topological quotient of ${\rm max}\, R$ via $\pi_m$ unless $R$ satisfies the $\Delta$-Dixmier-Moeglin equivalence.

\begin{proposition}\label{prop:DMECore}
Let $(R,\Delta)$ be a commutative noetherian differential $k$-algebra with both $\dim_k R$ and $|\Delta|<|k|$. Then the following are equivalent:
\begin{enumerate}
\item $R$ satisfies the $\Delta$-Dixmier-Moeglin equivalence.
\item The prime differential core of any $P\in \Delta\text{-}{\rm prim}\, R$ is locally closed in ${\rm spec}\, R$.
\item The differential core of any $P\in \Delta\text{-}{\rm prim}\, R$ is locally closed, and $$\overline{\mathscr C(P)}~=~\left\{Q\in {\rm max}\, R\,|\, P\subseteq Q\right\}$$ in ${\rm max}\, R$.
\end{enumerate}
If the above conditions hold, then each (prime) differential core of $\Delta$-primitive ideals is smooth in its closure.
\end{proposition}
\begin{proof}
(i)$\Rightarrow$(ii) We know $P$ is locally closed in $\Delta$-${\rm spec}\, R$. Since $\Delta$-${\rm spec}\, R$ is a topological quotient of ${\rm spec}\, R$ via $\pi$, we have $\pi^{-1}(P)=\mathscr C_{\rm spec}(P)$ is locally closed in ${\rm spec}\, R$

(ii)$\Rightarrow$(i) Since $P\in \mathscr C_{\rm spec}(P)$, we know that $\overline{\mathscr C_{\rm spec}(P)}=\{Q\in {\rm spec}\, R\,|\, P\subseteq Q\}$.
By Lemma \ref{lem:equilocallyclosed}, that $\mathscr C_{\rm spec}(P)$ is locally closed implies that there is some ideal $I$ properly containing $P$ such that
$$\{Q\in {\rm spec}\, R\,|\, P\subseteq Q\}\setminus \{Q\in {\rm spec}\, R\,|\, (Q:\Delta)=P\}=\{Q\in {\rm spec}\, R\,|\, I\subseteq Q\}.$$
Let $Q$ be any $\Delta$-prime ideal properly containing $P$. Since $P\subsetneq (Q:\Delta)=Q$, we have $I\subseteq Q$. By Lemma \ref{lem:equilocallyclosed}, we know $P$ is locally closed in $\Delta\text{-}{\rm spec}\, R$. Hence $R$ satisfies the $\Delta$-Dixmier-Moeglin equivalence by Theorem \ref{thm:main1}.

(i)$\Rightarrow$(iii) Let $I$ be the intersection of all $\Delta$-prime ideals properly containing $P$. Since $P$ is locally closed in $\Delta$-${\rm spec}\, R$, we have $P\subsetneq I$ and
\begin{multline*}
\mathscr C(P)=\{Q\in {\rm max}\, R\,|\, P\subseteq (Q:\Delta)\}\setminus \{Q\in {\rm max}\, R\,|\, P\subsetneq (Q:\Delta)\}\\
=\{Q\in {\rm max}\, R\,|\, P\subseteq Q\}\setminus \{Q\in {\rm max}\, R\,|\, I\subseteq Q\}
\end{multline*}
So $\mathscr C(P)$ is locally closed in ${\rm max}\, R$. Moreover since $R$ is a Jacobson ring, we have
$$
P~=~\bigcap_{P\subseteq Q\in {\rm max}\, R}\ Q~=~\left(\bigcap_{Q\in \mathscr C(P)}\ Q\right)\bigcap \left(\bigcap_{I\subseteq Q\in {\rm max}\, R}\ Q\right).
$$
Since $P$ is prime and $P\subsetneq I\subseteq \bigcap_{I\subseteq Q\in {\rm max}\, R} Q$, we have $\bigcap_{Q\in \mathscr C(P)} Q=P$ and 
$$\overline{\mathscr C(P)}~=~\left\{Q\in {\rm max}\, R\,\left |\, \left(\bigcap_{M\in \mathscr C(P)}\ M\right)\subseteq Q\right\}\right.~=~\left\{Q\in {\rm max}\, R\,|\, P\subseteq Q\right\}.$$

(iii)$\Rightarrow$(i) By Theorem \ref{thm:main1}, it suffices to show that any $\Delta$-primitive ideal $P$ of $R$ is locally closed in $\Delta\text{-}{\rm spec}\, R$. Since $\mathscr C(P)$ is open in $\overline{\mathscr C(P)}=\{Q\in {\rm max}\, R\,|\, P\subseteq Q\}$, there is some $\Delta$-ideal $I$ properly containing $P$ such that
$$\overline{\mathscr C(P)}\setminus \mathscr C(P)~=~\{Q\in {\rm max}\, R\,|\, P\subsetneq (Q:\Delta)\}~=~\{Q\in {\rm max}\, R\,|\, I\subseteq Q\}.$$
For any $\Delta$-prime ideal $L$ properly containing $P$, by Lemma \ref{DeltaP}(iv), we have
$$
L~=~\bigcap_{L\subseteq Q\in \Delta\text{-}{\rm prim}\, R}\ Q~=~\bigcap_{L\subseteq M\in {\rm max}\, R}\ (M:\Delta)~\supseteq~(I:\Delta)~=~I~\supsetneq~P.
$$
So $P$ is locally closed in $\Delta\text{-}{\rm spec}\, R$.

Finally, let $J$ be the defining ideal of the singular locus of $\overline{\mathscr C(P)}$ (resp. $\overline{\mathscr C_{\rm spec}(P)}$). Since $J$ is a $\Delta$-ideal properly containing $P$, it follows from the definition of $\Delta$-core that any maximal ideal $Q$ in $\mathscr C(P)$  does not contain $J$. Hence the point corresponding to $Q$ belongs to the smooth locus of $\overline{\mathscr C(P)}$, as required.
\end{proof}

Finally, suppose $G$ is an affine algebraic group over $k$. Let $G$ act on $R$ by $k$-algebra automorphisms. An {\it $G$-ideal} of $R$ is any ideal $I$ of $R$ that is stable under $G$. For any ideal $I$ in $R$, let $(I:G)$ denote the largest $G$-ideal of $R$ contained in $I$, that is,
$$(I:G)~=~\bigcap_{g\in G}\ g(I)~=~\left\{r\in R\,|\, g(r)\in I\ ,\forall \ g\in G\right\}.$$
A proper $G$-ideal $I$ of $R$ is called {\it $G$-prime} if $JK\subseteq I$ for any two $G$-ideals $J,K$ of $R$ implies that either $J\subseteq I$ or $K\subseteq I$. In particular, $(P:G)$ is $G$-prime for any prime ideal $P$. A $(G,\Delta)$-ideal of $R$ is any ideal that is both a $G$-ideal and a $\Delta$-ideal. We recall the notion of {\it $(G, \Delta)$-prime ideal} where the primeness is defined in the reign of all $(G,\Delta)$-ideals, and we write $(G, \Delta)$-${\rm spec}\, R$ for the $(G, \Delta)$-spectrum of $R$, that is, the set of all $(G, \Delta)$-prime ideals of $R$, equipped with the natural Zariski topology.

Now we introduce the concept of rational action of $G$ on $(R,\Delta)$; it generalizes the notion of a rational torus action on $(R,\Delta)$ in \cite{Go06}.

\begin{definition}\label{def:rational}
Let $(R, \Delta)$ be a differential $k$-algebra, and $G$ an affine algebraic group over $k$. Suppose $G$ acts on $R$ by $k$-algebra automorphisms. The action of $G$ on $(R,\Delta)$ is {\it rational} if
\begin{enumerate}
\item $R$ is a directed union of finite-dimensional $G$-invariant $k$-subspaces $V_i$ such that the restriction maps $G\to {\rm Aut}(R)\to {\rm GL}(V_i)$ are morphisms of algebraic varieties.
\item The $k$-linear span of $\Delta$ is a directed union of finite dimensional $G$-invariant $k$-subspaces of $\Der_k\, R$ under the induced $G$-action such that $g.\delta=g\circ \delta\circ g^{-1}$ for any $\delta\in \Der_k\, R$ and $g\in G$.
\end{enumerate}
\end{definition}

The following lemma is based on \cite[Lemma 3.1]{Go06}, where we extend any rational action of an algebraic torus to that of an affine connected algebraic group.

\begin{lemma}\label{lem:GPideal}
Let $(R,\Delta)$ be a commutative noetherian differential $k$-algebra and $G$ an affine algebraic group acting rationally on $(R,\Delta)$.
\begin{enumerate}
\item If $I$ is a $\Delta$-ideal of $R$, then $(I:G)$ is a $(G,\Delta)$-ideal. Similarly, if $I$ is a $G$-ideal, then $(I:\Delta)$ is a $(G,\Delta)$-ideal.
\item $(P:G)\in (G, \Delta)$-${\rm spec}\, R$  for al $P\in \Delta$-${\rm spec}\, R$.
\item Assume that $G$ is connected. The following sets coincide:
\begin{enumerate}
\item $(G,\Delta)$-${\rm spec}\, R$;
\item The set of all $G$-prime $\Delta$-ideals in $R$;
\item The set of all $\Delta$-prime $G$-ideals in $R$;
\item The set of all prime $(G,\Delta)$-ideals in $R$.
\end{enumerate}
\end{enumerate}
\end{lemma}
\begin{proof}
(i) Let $I$ be a $\Delta$-ideal of $R$. Then we have
$$g(\Delta(I:G))~=~g\circ \Delta\circ g^{-1}(g(I:G))~\subseteq ~g.\Delta(I)~\subseteq ~\Delta(I)~\subseteq ~I$$
for all $g\in G$. So $\Delta(I:G)\subseteq (I:G)$ and $(I:G)$ is a $(G,\Delta)$-ideal. Similarly for any $G$-ideal $I$, we have
\begin{align*}
\delta_1\cdots \delta_ng(I:\Delta)~&=~g(g^{-1}\delta_1g)\cdots (g^{-1}\delta_ng)(I:\Delta)~=~g(g^{-1}.\delta_1)\cdots (g^{-1}.\delta_n)(I:\Delta)\\
~&\subseteq~g(I:\Delta)~\subseteq~g(I)~\subseteq~I
\end{align*}
for any $\delta_1,\ldots,\delta_n\in \Delta$ and $g\in G$. This implies that $g(I:\Delta)\subseteq (I:\Delta)$ for any $g\in G$. So $(I:\Delta)$ is a $(G,\Delta)$-ideal.

(ii) For any $P\in \Delta$-${\rm spec}\, R$, it is clear that $(P:G)$ is a $G$-prime ideal. By (i), $(P:G)$ is a $(G,\Delta)$-ideal, and we can see that it is also $(G,\Delta)$-prime since it is  $G$-prime.

(iii) By \cite[Corollary 1.3]{Chin}, any $G$-prime ideal of $R$ is prime. Moreover, by \cite[Lemma 1.1(d)]{Go06}, any $\Delta$-prime ideal of $R$ is also prime. So the sets (b), (c), and (d) coincide. It is clear that (d) $\subseteq$ (a).

Finally, let $Q$ be a $(G,\Delta)$-prime ideal. By Lemma \ref{DeltaP}(iii), there are only finite many minimal prime ideals over $Q$ and all of them are $\Delta$-ideals.  Say $Q_1,\ldots,Q_n$ are those minimal primes over $Q$. Since $G$ must permute those $Q_i$'s, we know their $G$-orbits are finite. Then \cite[Proposition II.2.9]{BrGo} implies that $Q_1,\ldots,Q_n$ are all $G$-ideals and hence are all $(G,\Delta)$-ideals. Since $Q_1\cdots Q_n\subseteq Q$, we have $Q=Q_i$ for some $i$ since $Q$ is $(G,\Delta)$-prime. Therefore (a) $\subseteq$ (d).
\end{proof}

As a consequence, there exists a corresponding $G$-stratification of $\Delta$-${\rm spec}\, R$ such that
$$
\Delta\text{-}{\rm spec}\, R~=~\bigsqcup_{J\in (G, \Delta)\text{-}{\rm spec\, R}}\ \Delta\text{-}{\rm spec}_J\,R,
$$
where
$$\Delta\text{-}{\rm spec}_J\,R~=~\left\{P\in \Delta\text{-}{\rm spec}\, R : (P:G)=J\right\}\ \text{for all}\ J\in (G, \Delta)\text{-}{\rm spec}\, R.$$
Now suppose $(G, \Delta)\text{-}{\rm spec}\, R$ is finite. We can list the $(G,\Delta)$-prime ideals as $J_1,\ldots,J_n$. Therefore, we have
$$\Delta\text{-}{\rm spec}_{J_i}R~=~\{P\in \Delta\text{-}{\rm spec}\, R: P\supseteq J_i\}\setminus \bigcup_{J_i\subsetneqq J_j} \{P\in \Delta\text{-}{\rm spec}\, R: P\supseteq J_j\}.$$
Hence each $\Delta\text{-}{\rm spec}_{J_i}R$ is locally closed in $\Delta\text{-}{\rm spec}\, R$. So we can apply our local criterion of $\Delta$-Dixmier-Moeglin equivalence.

\begin{proposition}\label{prop:finiteGprime}
Let $(R,\Delta)$ be a commutative noetherian differential $k$-algebra with both $\dim_k R$ and $|\Delta|<|k|$, and $G$ an affine connected algebraic group acting rationally on $(R,\Delta)$. Suppose the set $(G, \Delta)\text{-}{\rm spec}\, R$ is finite and for each $J\in (G, \Delta)\text{-}{\rm spec}\, R$, we have $\Delta\text{-}{\rm spec}_J\,R$ is homeomorphic to ${\rm spec}\, R_J$ for some noetherian commutative $k$-algebra $R_J$ with $\dim_k R_J<|k|$. Then $R$ satisfies the $\Delta$-Dixmier-Moeglin equivalence.
\end{proposition}
\begin{proof}
It follows from Theorem \ref{thm:stratification} together with the discussion above.
\end{proof}

\begin{example}
Let $G=(k^\times)^r$ be an algebraic torus acting rationally on a commutative noetherian differential $k$-algebra $(R,\Delta)$ with both $\dim_k R$ and $|\Delta|<|k|$. By \cite[Theorem 3.2]{Go06}, we known each $\Delta\text{-}{\rm spec}_{J}R$ is homeomorphic to ${\rm spec}\, R_J$, where $R_J$ is a Laurent polynomial ring in at most $r$ indeterminates over the fixed field $Z_\Delta(\Frac R/J)^G$. Now suppose $R$ has only finitely many $(G,\Delta)$-ideals. Our Proposition \ref{prop:finiteGprime} then yields that $R$ satisfies the $\Delta$-Dixmier-Moeglin equivalence with a slightly different assumption on $(R,\Delta)$ compared to \cite[Theorem 3.3]{Go06} .
\end{example}

\section{Applications to the Poisson Dixmier-Moeglin equivalence}\label{S:P}
In our last section, we give some applications to complex affine Poisson algebras. Let $(A,\{-,-\})$ be an affine complex Poisson algebra. We can see that $(A,\Delta_A)$ is indeed a commutative differential $\mathbb C$-algebra with $\Delta_A:=\{\{a,-\}\,|\, a\in A\}$ consisting of all Hamiltonian derivations of $A$. Since $A$ is finitely generated algebra over $\mathbb C$, we can always choose $\Delta_A$ to be a finite set only consisting of those Hamiltonian derivations corresponding to a finite set of generators of $A$. Moreover, $\Delta$-ideals, $\Delta$-prime ideals, and $\Delta$-primitive ideals of $A$ are exactly Poisson ideals, Poisson prime ideals, and Poisson primitive ideals of $A$, respectively. In particular, $\Delta$-${\rm spec}\,A={\rm P. spec}\, A$ is the Poisson prime spectrum of $A$ and $\Delta$-${\rm prim}\,A={\rm P. prim}\, A$ is the subset of ${\rm P. spec}\, A$ consisting of Poisson primitive ideals with the subspace topology. We use $\mathcal P (I):=(I:\Delta_A)$ to denote the Poisson core of any ideal $I$ in $A$. Finally, the $\Delta$-Dixmier-Moeglin equivalence holds for the commutative differential $\mathbb C$-algebra $(A,\Delta_A)$ if and only if the Poisson Dixmier-Moeglin equivalence holds for the Poisson algebra $A$.

Our main result Theorem \ref{thm:rationalprimitive} yields another proof of \cite[Theorem 3.2]{BLLM17}.

\begin{theorem}\label{thm:Prationalprimitive}
Let $A$ be a complex affine Poisson algebra. For any Poisson prime ideal $P$, the following are equivalent:
\begin{enumerate}
\item $P$ is Poisson primitive;
\item $P$ is Poisson rational;
\item $P$ is $|\mathbb C|$-separable in ${\rm P. spec}\ A$.
\end{enumerate}
\end{theorem}
\begin{proof}
As discussed above, $(A,\Delta_A)$ is a commutative differential complex affine algebra, where $\Delta_A$ can be chosen as a finite set. Hence we know $|\Delta_A|\le \dim_\mathbb C\,A\le \aleph_0<\aleph=|\mathbb C|$. As a consequence, the equivalences follow from Theorem \ref{thm:rationalprimitive}.
\end{proof}

Our first topological description of the Poisson Dixmier-Moeglin equivalence for any complex affine Poisson algebra $A$ is given in terms of the poset $({\rm P. spec}\, A,\subseteq)$. 

\begin{theorem}
\label{thm:Poissonmain1}
Let $A$ be a complex affine Poisson algebra. Then the following are equivalent.
\begin{enumerate}
\item $A$ satisfies the Poisson Dixmier-Moeglin equivalence.
\item Every Poisson primitive ideal of $A$ is locally closed in ${\rm P. spec}\, A$.
\item Every Poisson rational ideal of $A$ is locally closed in ${\rm P. spec}\, A$.
\item In the poset $({\rm P. spec}\, A,\subseteq)$, that a Poisson prime ideal is $|\mathbb C|$-separable implies that it is $\aleph_0$-separable.
\end{enumerate}
\end{theorem}
\begin{proof}
As discussed above, $(A,\Delta_A)$ is a commutative differential complex affine algebra, where $\Delta_A$ can be chosen as a finite set. Hence we know $|\Delta_A|\le \dim_\mathbb C\,A\le \aleph_0<\aleph=|\mathbb C|$. As a consequence, the equivalences follow from Theorem \ref{thm:main1}.
\end{proof}

\begin{theorem}
\label{thm:Poissonmain2}
Let $A$ be a complex affine Poisson algebra. Suppose ${\rm P. spec}\, A=\bigcup X_i$ is a union of locally closed subsets with each $X_i$, when endowed with the subspace topology, homeomorphic to ${\rm P. spec}\, A_i$ for some complex affine Poisson algebra $A_i$. Then $A$ satisfies the Poisson Dixmier-Moeglin equivalence if and only if each $A_i$ satisfies the Poisson Dixmier-Moeglin equivalence. In particular, if each $X_i$ is homeomorphic to ${\rm spec}\, A_i$ for some commutative noetherian complex algebra essentially of finite type, then $A$ satisfies the Poisson Dixmier-Moeglin equivalence.
\end{theorem}
\begin{proof}
It is Theorem \ref{thm:stratification} and Corollary \ref{cor:commutative} with the suitable assumptions on $A$ and $A_i$'s.
\end{proof}

We have two different stratifications on the maximal spectrum $\mathscr A$ of a complex affine Poisson algebra $A$. One is the symplectic core stratification given as a special case of the differential core stratification on $(A,\Delta_A)$. 

The equivalence relation $\sim$ on ${\rm spec}\, A$ defined by the symplectic core stratification is given by:
$$
\mathfrak p\sim \mathfrak q~\Longleftrightarrow~\mathcal P(\mathfrak p)=\mathcal P(\mathfrak q),
$$
where we denote the corresponding equivalence class of any $\mathfrak p\in {\rm spec}\, A$ by $\mathscr C_{\rm spec}(\mathfrak p)$, so that
$${\rm spec}\, A~=~ \bigsqcup_{\mathfrak q\in {\rm P. spec}\, A}\quad \mathscr C_{\rm spec}(\mathfrak q)~=~\left\{\mathfrak p\in {\rm spec}\, A\,|\, \mathcal P(\mathfrak p)=\mathfrak q\right\}.$$
The set $\mathscr C_{\rm spec}(\mathfrak p)$ is called the {\it prime symplectic core} (of $\mathfrak p$). Moreover, the {\it (maximal) symplectic core} of $\mathfrak p$ is denoted by $\mathscr C(\mathfrak p):=\mathscr C_{\rm spec}(\mathfrak p)\bigcap \mathscr A$ and we have
$$\mathscr A~=~ \bigsqcup_{\mathfrak q\in {\rm P. prim}\, A}\quad \mathscr C(\mathfrak q)~=~\left\{\mathfrak m\in \mathscr A\,|\, \mathcal P(\mathfrak m)=\mathfrak q\right\}.$$

The other one is called the symplectic leaf stratification. When $A$ is smooth, the symplectic leaf stratification on $\mathscr A$ is the symplectic foliation on the complex analytic Poisson manifold $\mathscr A$ in the sense of \cite{Wein83}. The general case for $A$ not being smooth is discussed in \cite[\S 3.5]{BrGor03}. 

\begin{definition}\cite[\S 3.6]{BrGor03}\label{D:Palg}
The Poisson bracket of $A$ is said to be {\it algebraic} if each symplectic leaf appearing in the symplectic leaf stratification is a locally closed subvariety of $\mathscr A$. 
\end{definition}

The following results are well-known. We refer to \cite[Lemma 3.5]{BrGor03} and \cite[Proposition 3.6]{BrGor03} for details of proofs.  

\begin{lemma}\label{lem:slc}
Let $A$ be a complex affine Poisson algebra, and $\mathscr L(\mathfrak m)$ be the symplectic leaf containing some maximal ideal $\mathfrak m\in \mathscr A$. Then we have
\begin{enumerate}
\item $\overline{\mathscr L(\mathfrak m)}=\{\mathfrak q\in \mathscr A\,|\, \mathcal P(\mathfrak m)\subseteq \mathfrak q\}$;
\item  $\mathscr L(\mathfrak m)\subseteq \mathscr C(\mathfrak m)$ and $\overline{\mathscr L(\mathfrak m)}=\overline{\mathscr C(\mathfrak m)}$;
\item Each symplectic core is a disjoint union of symplectic leaves;
\item The Zariski closure of each symplectic leaf (resp. symplectic core) is a disjoint union of symplectic leaves (resp. symplectic cores).
\end{enumerate}
Further suppose the Poisson bracket of $A$ is algebraic. Then the symplectic core stratification coincides with the symplectic leaf stratification.
\end{lemma}

There is a continuous retraction $\pi: {\rm spec}\, A\to {\rm P. spec}\, A$ given by the Poisson core such that $\pi(\mathfrak p)=\mathcal P(\mathfrak p)$ for any $\mathfrak p\in {\rm spec}\, A$. Moreover, ${\rm P. spec}\, A$ is a topological quotient of ${\rm spec}\, A$ via $\pi$ \cite[Theorem 1.3]{Go06}. The restriction of $\pi$ on ${\rm max}\, A$ gives a continuous surjection denoted by $\pi_m: {\rm max}\, A\to {\rm P. prim}\, A$ \cite[Theorem 1.5]{Go06}.

The following proposition gives an affirmative answer to \cite[Question 1.4]{Go06} asking whether or not the $\Delta$-primitive spectrum is a topological quotient of the maximal spectrum for any commutative differential algebra in the special case of complex affine Poisson algebra.

\begin{proposition}\label{prop:topquot}
Let $A$ be a complex affine Poisson algebra. Then ${\rm P. prim}\, A$ is a topological quotient of ${\rm max}\, A$ via $\pi_m: {\rm max}\, A\to {\rm P. prim}\, A$.
\end{proposition}
\begin{proof}
By definition, it is routine to check that $\pi_m: \mathscr A\to {\rm P. prim}\, A$ is a continuous surjective map. Now suppose $\pi_m^{-1}(X)$ is closed in $\mathscr A$ for some subset $X\subseteq {\rm P. prim}\, A$. In order to show that ${\rm P. prim}\, A$ is a topological quotient of $\mathscr A$ via $\pi_m$, we need to show that $X$ is closed in ${\rm P. prim}\, A$. We know there exists some ideal $I$ of $A$ such that $\pi_m^{-1}(X)=\{\mathfrak m\in \mathscr A\,|\, I\subseteq \mathfrak m\}$. It suffices to show that $X=\{\mathfrak m\in {\rm P. prim}\, A\,|\, I\subseteq \mathfrak m\}$.

In one direction, take any $\mathfrak q\in {\rm P. prim}\, A$ such that $I\subseteq \mathfrak q$. Write $\mathfrak q=\mathcal P(\mathfrak m)$ for some $\mathfrak m\in \mathscr A$. It is clear that $I\subseteq \mathfrak q=\mathcal P(\mathfrak m)\subseteq \mathfrak m$. So $\mathfrak m\in \pi_m^{-1}(X)$ and $\mathfrak q=\pi_m(\mathfrak m)\in \pi_m(\pi_m^{-1}(X))=X$. So $X\supseteq \{\mathfrak m\in {\rm P. prim}\, A\,|\, I\subseteq \mathfrak m\}$.

In the other direction, take any $\mathfrak q\in X$. It is clear that $\pi_m^{-1}(\mathfrak q)=\mathscr C(\mathfrak q)\subseteq \pi_m^{-1}(X)$. By Lemma \ref{lem:slc}(i)\&(ii), we get
$$\left\{\mathfrak m\in \mathscr A\,|\, q\subseteq \mathfrak m\right\}~=~\overline{\mathscr C(\mathfrak q)}~\subseteq~\pi_m^{-1}(X)~=~\left\{\mathfrak m\in \mathscr A\,|\, I\subseteq \mathfrak m\right\}.$$
This implies that $I\subseteq \mathfrak q$ and $\mathfrak q\in  \{\mathfrak m\in {\rm P. prim}\, A\,|\, I\subseteq \mathfrak m\}$. So $X\subseteq \{\mathfrak m\in {\rm P. prim}\, A\,|\, I\subseteq \mathfrak m\}$. This completes our proof.
\end{proof}

Our second topological description of the Poisson Dixmier-Moeglin equivalence for any complex affine Poisson algebra $A$ is given in terms of the symplectic core stratification on the maximal spectrum $\mathscr A$. 

\begin{theorem}\label{thm:algebraicDME}
Let $A$ be a complex affine Poisson algebra. Then the following are equivalent.
\begin{enumerate}
\item $A$ satisfies the Poisson Dixmier-Moeglin equivalence;
\item $\mathscr C_{\rm spec}(\mathfrak p)$ is locally closed in ${\rm spec}\, A$ for any $\mathfrak p\in {\rm spec}\, A$;
\item $\mathscr C(\mathfrak m)$ is locally closed in ${\rm max}\, A$ for any $\mathfrak m\in {\rm max}\, A$.
\end{enumerate}
\end{theorem}
\begin{proof}
The equivalences come from Proposition \ref{prop:DMECore} by considering the complex affine Poisson algebra $A$ as a special commutative differential algebra $(A,\Delta_A)$. In particular, for (iii) the condition $\overline{\mathscr C(\mathfrak m)}=\{\mathfrak q\in \mathscr A\,|\, \mathcal P(\mathfrak m)\subseteq \mathfrak q\}$ is automatically satisfied by Lemma \ref{lem:slc}(i)\&(ii).
\end{proof}

\begin{corollary}\label{cor:algebraicDME}
Let $A$ be a complex affine Poisson algebra. If the Poisson bracket of $A$ is algebraic, then $A$ satisfies the Poisson Dixmier-Moeglin equivalence. 
\end{corollary}
\begin{proof}
The result follows from Lemma \ref{lem:slc} and Theorem \ref{thm:algebraicDME}.
\end{proof}

\begin{remark}\label{rem:algebraic}
In general, the inverse statement of Corollary \ref{cor:algebraicDME} does not hold. The following example is taken from \cite[Example 2.37]{Van96} (also see \cite[Remarks \S 3.6(1)]{BrGor03}). Let $A=\mathbb C[x,y,z]$ with Poisson bracket given by $\{x,y\}=0, \{x,z\}=\alpha x$ and $\{y,z\}=y$ for some $\alpha\in \mathbb R$. Then the symplectic leaves have small dimensions: the single points $(0,0,c)$ for $c\in \mathbb C$ and the 2-dimensional symplectic leaves defined by the equation $xy^\alpha=\text{constant}$. If $\alpha$ is rational, then the leaves are algebraic; but when $\alpha$ is irrational, then the leaves are not algebraic varieties. On the other side, the symplectic cores are given by points $(0,0,c)$ for $c\in \mathbb C$ and $(\mathbb C^\times)^2\times \mathbb C$. So they are always locally closed. 
\end{remark}

Let $G$ be an algebraic group. Suppose that $G$ acts rationally by Poisson automorphisms on $A$. It is easy to check that $G$ then acts rationally on the commutative differential algebra $(A,\Delta_A)$ where $\Delta_A$ can be chosen to be a finite set of Hamiltonian derivations $\{a_i,-\}_{1\le i\le n}$ where $\{a_i\}_{1\le i\le n}$ is a basis for some finite-dimensional $G$-invariant generating space of $A$. As a consequence, $G$ acts on the symplectic leaves and symplectic cores of $A$, and we label the $G$-orbits of symplectic leaves (or symplectic cores) $\mathscr G_i$, so that we have a stratification
$$
\mathscr A~=~\bigsqcup\ \mathscr G_i,
$$
where $\mathscr G_i=\sqcup_{g\in G/{\rm stab}(\mathscr L)}\, g.\mathscr L$ for some symplectic leaf (or symplectic core) $\mathscr L$ of $\mathscr A$. The following result is a $G$-equivariant version of \cite[Proposition 3.7]{BrGor03}.

\begin{theorem}\label{thm:GalgebraDME}
Let $G$ be an algebraic group acting rationally by Poisson automorphism on a complex affine Poisson algebra $A$. If the $G$-equivariant stratification of $A$ into $G$-orbits of symplectic leaves (or symplectic cores) is finite, then $A$ satisfies the Poisson Dixmier-Moeglin equivalence.
\end{theorem}
\begin{proof}
According to Lemma \ref{lem:slc}(iii), any $G$-orbit of symplectic cores is a disjoint union of $G$-orbits of symplectic leaves. So if there are only finitely many $G$-orbits of symplectic leaves, there are only finitely many $G$-orbits of symplectic cores. So without loss of generality, we may assume that there are only finitely many $G$-orbits of symplectic cores, which are listed as follows
$$\mathscr G_1,\ \mathscr G_2,\ \ldots,\ \mathscr G_t.$$
We argue by noetherian induction to prove that any symplectic core contained in $\mathscr G_i$ is locally closed. We may assume that any symplectic core contained in $\mathscr G_i$, where $\overline{\mathscr G_i}\neq \mathscr A$, is locally closed. Without loss of generality, suppose $\mathscr G_1,\ldots,\mathscr G_n$ are the symplectic cores whose Zariski closure is exactly $\mathscr A$. Write $X=\sqcup_{1\le i\le n}\, \mathscr G_i$. We can assume that $\mathscr A$ is irreducible. We claim that $X$ is a non-empty open set of $\mathscr A$. Notice that we can prove a $G$-equivariant version of Lemma \ref{lem:slc}(iv) such that each $\overline{\mathscr G_i}$ is a disjoint union of $G$-orbits of symplectic cores contained in it. Suppose $\mathscr G_i=\bigcup_{g\in G} g.\mathscr C$ for some symplectic core $\mathscr C$. By Lemma \ref{lem:slc}, we know $\overline{\mathscr C}=\{\mathfrak m\in \mathscr A\,|\, \mathfrak p\subseteq \mathfrak m\}$ for some Poisson primitive ideal $\mathfrak p$. Then it is easy to see that, for any $g\in G$, we have $\overline{g.\mathscr C}= \{\mathfrak m\in \mathscr A\,|\, g(\mathfrak p)\subseteq \mathfrak m\}$. So
$$\overline{\mathscr G_i}~=~\overline{\bigcup_{g\in G}\, g.\mathscr C}~=~\left\{\mathfrak m\in \mathscr A\,\left|\, \left(\bigcap_{g\in G} g(\mathfrak p)\right)\subseteq \mathfrak m\right\}\right.~=~\left\{\mathfrak m\in \mathscr A\,\left|\, (\mathfrak p:G)\subseteq \mathfrak m\right\}\right..$$
Since $(\mathfrak p:G)$ is a Poisson $G$-ideal by Lemma \ref{lem:GPideal}(i), for any $\mathfrak m\in \mathscr A$, we know $\mathfrak m\in\overline{\mathscr G_i}$ if and only if $G.\mathscr C(\mathfrak m)\subseteq \overline{\mathscr G_i}$. So $\overline{\mathscr G_i}$ is a disjoint union of $G$-orbits of symplectic cores. Therefore if $\overline{\mathscr G_i}\neq \mathscr A$, we have $\overline{\mathscr G_i}$ can not contain any $\mathscr G_1,\ldots,\mathscr G_n$ whose closure is $\mathscr A$. Then we know
$$X~=~\mathscr A\setminus\left(\bigsqcup_{\overline{\mathscr G_i}\neq \mathscr A}\ \mathscr G_i\right)~=~\mathscr A\setminus\left(\bigcup_{\overline{\mathscr G_i}\neq \mathscr A}\ \overline{\mathscr G_i}\right).$$
Since there are only finitely many $\mathscr G_i$, we see $X$ is a non-empty open set of $\mathscr A$. This proves our claim.

It remains to show that any symplectic core $\mathscr C$ in $X$ is locally closed. Since $X$ is open, it suffices to show that $\mathscr  C=\overline{\mathscr C}\cap X$. One direction $\mathscr C\subseteq \overline{\mathscr C}\cap X$ is clear. For the other direction, take any $\mathfrak m\in X$ whose rank is maximal among $X$, say ${\rm rk}(\mathfrak m)=j$ and $\mathfrak m\in \mathscr G_1$. Let $\mathscr B$ be the proper closed subvariety of $\mathscr A$ consisting of points whose rank is strictly less than $j$. Since the action of $G$ preserves the rank of point, it is clear that $\mathscr B$ is $G$-invariant. So $\mathscr B$ is a disjoint union of $G$-orbits of symplectic cores by \cite[Proposition 3.6(1)]{BrGor03}. Moreover since $\overline{\mathscr G_i}=\mathscr A$, we have $\mathscr G_i \not\subseteq \mathscr B$ for all $1\le i\le n$. Then we get $X$ is an irreducible open subset of $\mathscr A$ on which the rank is constant and equal to $j$. As a consequence, the closure of any symplectic core contained in $X$ is $j$-dimensional. In particular, $\mathscr C$ is $j$-dimensional. But by Lemma \ref{lem:slc}, one can see that $\overline{\mathscr C}\setminus \mathscr C$ is a disjoint union of symplectic cores whose closures all have dimension strictly less than $j$. Hence we have
$$(\overline{\mathscr C}\cap X)\setminus \mathscr C~\subseteq~(\overline{\mathscr C}\setminus \mathscr C)~\bigcap~\left(X\setminus \mathscr C\right)=\emptyset.$$
So $\mathscr C=\overline{\mathscr C}\cap X$ is locally closed.

Finally, our result follows from Theorem \ref{thm:algebraicDME}.
\end{proof}

%

Now we give some examples by using our main results on stratifications of maximal spectrum of complex affine Poisson algebras.

\begin{example}
Let $S$ be a three-dimensional Sklyanin algebra over the complex numbers. By definition, $S$ is a connected graded noncommuative algebra generated by three indeterminants  $x,y,z$ and subject to three quadratic relations:
$$axy+byx+cz^2~=~ayz+bzy+cx^2~=~azx+bxz+cy^2~=~0$$
for some $[a:b:c]\in \mathbb P^2$ satisfying $abc\neq 0$ and $(3abc)^3\neq (a^3+b^3+c^3)^3$. In \cite{ATV90}, it is shown that $S$ has a central regular element $g$ of degree three and $S/gS\cong B(E,\mathscr L,\sigma)$, the twisted homogeneous coordinate ring of the associated geometric data $(E,\mathscr L,\sigma)$ of $S$. It is further shown that $S$ is module-finite over its center if and only if $|\sigma|<\infty$, where $\sigma$ is some translation on the elliptic curve $E$. 

In the following, we assume that $S$ is module-finite over its center $Z$ and denote $n=|\sigma|<\infty$. We know by \cite[Proposition 5.5(1)]{WWY17} that there exists a unique nonzero Poisson structure (up to scalars) on $Z$ if $g$ is assumed to be in the Poisson center of $Z$. More precisely, according to \cite{ST94}, the center $Z$ of $S$ can be presented as
$$Z=\mathbb C[z_1,z_2,z_3,g]/(F),$$
where $z_1,z_2,z_3$ are three algebraically independent central elements of degree $n$ and $F$ is a single relation of degree $3n$. The Poisson bracket on $Z$ is then determined (up to scalars) as
$$
\{z_1,z_2\}=\partial_{z_3}F,\quad \{z_2,z_3\}=\partial_{z_1}F,\quad\{z_3,z_1\}=\partial_{z_2}F,
$$
with $g$ in the Poisson center of the Poisson algebra $Z$. As a consequence of \cite[Theorem 1.3(2)]{WWY17}, we know the Poisson bracket of $Z$ is algebraic. Moreover, the group $\Sigma:=\mathbb Z_3\times \mathbb C^\times$ acts on $Z$ by Poisson automorphisms where $\mathbb Z_3$ permutes $z_1,z_2,z_3$ and $\mathbb C^\times$ scales all the variables $z_1,z_2,z_3,g$ simultaneously. It is shown also in \cite[Theorem 1.3(3)]{WWY17} that the $\Sigma$-orbits of symplectic cores (and hence symplectic leaves) are finite. Now by our Theorem \ref{thm:algebraicDME} or Theorem \ref{thm:GalgebraDME} we know $Z$ satisfies the Poisson Dixmier-Moeglin equivalence.
\end{example}

\begin{example}
Let $n\ge 2$ be an integer and consider $n-2$ generic polynomials $Q_i$  in $\mathbb C^n$ with coordinates $x_i$, $i=1,\ldots,n$. For any polynomial $\lambda\in \mathbb C[x_1,\ldots,x_n]$, we can define a bilinear differential operation
$$
\{-,-\}: \mathbb C[x_1,\ldots,x_n] \otimes \mathbb C[x_1,\ldots,x_n] \to \mathbb C[x_1,\ldots,x_n]
$$
by the formula
$$
\{f,g\}~=~\lambda\,\frac{df\wedge dg\wedge dQ_1\wedge \cdots \wedge dQ_{n-2}}{dx_1\wedge dx_2\wedge\cdots \wedge dx_n},\quad f,g\in  \mathbb C[x_1,\ldots,x_n].
$$
This operation gives a Poisson algebra structure on  $\mathbb C[x_1,\ldots,x_n]$. The polynomials $Q_i$, $i=1, \ldots, n-2$ are Casimir functions for the Poisson bracket  and any Poisson structure in $\mathbb C^n$ with $n-2$ generic Casimirs $Q_i$ are written in this form.

The case $n = 4$ corresponds to the classical (generalized) Sklyanin quadratic Poisson algebra. The very Sklyanin algebra is associated with the following two quadrics in $\mathbb C^4$:
\begin{align*}
Q_1~&=~x_1^2 + x_2^2 + x_3^2,\\
Q_2~&=~ J_{12}x_1^2 + J_{23}x_2^2 + J_{31}x_3^2+x_4^2,
\end{align*}
where we fix a lattice $\Lambda\subseteq \mathbb C$ and let $\{\theta_{ab}\,|\, ab=00,01,10,11\}$ denote Jacobi's four theta functions with period lattice $\Lambda$ such that
\begin{gather*}
J_{12}=\frac{\theta_{11}(\tau)^2\theta_{01}(\tau)^2}{\theta_{00}(\tau)^2\theta_{10}(\tau)^2},\ J_{23}=\frac{\theta_{11}(\tau)^2\theta_{10}(\tau)^2}{\theta_{00}(\tau)^2\theta_{01}(\tau)^2},\ J_{12}=-\frac{\theta_{11}(\tau)^2\theta_{00}(\tau)^2}{\theta_{01}(\tau)^2\theta_{10}(\tau)^2},
\end{gather*}
for some $\tau\in \mathbb C/\Lambda$. The Poisson brackets with $\lambda=1$ between the affine coordinates are described as follows
\begin{align*}
\{x_i,x_j\}~=~(-1)^{i+j}\, {\rm det}\left(\frac{\partial Q_k}{\partial x_l}\right),\ l\neq i,j\ \text{and}\ i>j.
\end{align*}
According to \cite[Proposition 2.2]{OR02}, these Poisson brackets all have symplectic leaves of small dimensions 0 or 2. More preciously, 2-dimensional symplectic leaves are the smooth locus of the complete intersections $(Q_1-\lambda_1,\ldots,Q_n-\lambda_n)$ excluding the hyperplane $\lambda=0$ and the 0-dimensional symplectic singletons are the points in the singluar locus of $(Q_1-\lambda_1,\ldots,Q_n-\lambda_n)$ besides the hyperplane $\lambda=0$ for all $(\lambda_1,\ldots,\lambda_n)\in \mathbb C^n$. As a consequence, the these Poisson brackets are all algebraic and the corresponding Poisson algebras $\mathbb C[x_1,\ldots,x_n]$ all satisfy the Poisson Dixmier-Moeglin equivalence.
\end{example}

\section*{Acknowledgments} Part of this research work was done during the first and second authors' visit to Shanghai Center for Mathematical Sciences in June-July 2019. They are grateful for the  invitations of the third author and wish to thank Fudan University for its hospitality. Juan Luo is supported by the NSFC (project 1190010260). Quanshui Wu is supported by the NSFC (project 11771085). The authors would also like to thank Jason Bell, Ken Brown, Ken Goodearl, and James Zhang for useful correspondence and comments on this paper. All the authors want to thank the referee for valuable comments.

\end{document}